\documentclass[11pt]{amsart}
\usepackage{amssymb}
\usepackage{amsthm}
\usepackage{mathrsfs}
\usepackage{aliascnt}
\usepackage{latexsym,amsfonts,amssymb,amsmath,amscd}
\usepackage{graphicx,color,ifpdf}
\usepackage{fancyhdr}
\allowdisplaybreaks
\usepackage[linktocpage, colorlinks, backref = page, citecolor = black, linkcolor = black]{hyperref}

\newcommand {\diam }{\mathrm{diam}}

\newcommand {\vol }{\mathrm{vol}}
\newcommand {\myd }{\;\mathrm{d}}
\newcommand {\ray }{\mathcal{R}}
\newcommand {\D }{\mathcal{D}}
\newcommand {\supp }{\mathrm{supp}}
\newcommand {\hess}{\mathrm{Hess}}
\newcommand {\R}{\mathbb{R}}
\newcommand {\lap }{\Delta}
\newcommand {\grad }{\triangledown}
\newcommand {\Ric}{\mathrm{Ric}}
\newcommand {\spa}[1]{\langle{#1}\rangle}
\newcommand {\myarrow}[1]{\mathop{\longrightarrow}\limits^{#1}}
\newcommand{\GH}{\myarrow{GH}}
\newcommand{\Xint}[1]{\mathchoice
	{\XXint\displaystyle\textstyle{#1}}%
	{\XXint\textstyle\scriptstyle{#1}}%
	{\XXint\scriptstyle
		\scriptscriptstyle{#1}}%
	{\XXint\scriptscriptstyle
		\scriptscriptstyle{#1}}%
	\!\int}
\newcommand{\XXint}[3]{{
		\setbox0=\hbox{$#1{#2#3}{\int}$}
		\vcenter{\hbox{$#2#3$}}\kern-.5\wd0}}
\newcommand{\dashint}{\Xint-}

\input xy
\xyoption{all}

\numberwithin{equation}{section}
\newtheorem{proposition}{Proposition}
\newtheorem{lemma}[proposition]{Lemma}
\newtheorem{sublemma}[proposition]{Sublemma}
\newtheorem{theorem}[proposition]{Theorem}
\newtheorem{corollary}[proposition]{Corollary}

\numberwithin{proposition}{section}

\theoremstyle{definition}
\newtheorem{definition}[proposition]{Definition}

\newtheorem{remark}[proposition]{Remark}

\title[Finite Topological Type]{A finite Topological Type Theorem for open manifolds with Non-negative Ricci Curvature and Almost Maximal Local Rewinding Volume}
\author{HongZhi Huang}
\address{ Department of Mathematics \\  Jinan University\\ Guangzhou 510632\\ PR China}
\email{\href{mailto:huanghz@jnu.edu.cn
	}{huanghz@jnu.edu.cn,
}{\href{mailto:hyyqsaax@163.com
}{hyyqsaax@163.com
}}}

\begin{document}

\maketitle
\begin{abstract} 
In this paper, we present finite topological type theorems for open manifolds with non-negative Ricci curvature, under almost maximal local rewinding volume. Unlike previous related research, our theorems remove the constraints of sectional curvature or conjugate radius, which were crucial additional assumptions on metric regularity in prior results. Notably, our settings do not necessarily satisfy a triangle comparison of Toponogov type. In fact, the method we adopt also extends to many previous related studies.

\vspace*{5pt}

\vspace*{5pt}

\end{abstract}

\section{Introduction}

An open (non-compact and complete) manifold $M$ is said to have finite topological type, if it is homeomorphic to the
interior of a compact manifold with boundary. According to Cheeger-Gromoll's soul theorem, an open manifold with non-negative sectional curvature always has finite topological type. However, this is not true if one relaxes the assumption on sectional curvature to Ricci curvature for the case of dimensions $\ge4$, as shown by Sha-Yang's example (\cite{SY}).

Therefore, it is natural to explore compelling additional assumptions that ensure the validity of a finite topological type theorem for open manifolds with non-negative Ricci curvature. In this direction, the first result is due to Abresch and Gromoll,

\begin{theorem}[\cite{AG}]\label{AGFiniteTpType}
	Let $M$ be an open $n$-Riemannian manifold with non-negative Ricci curvature. Suppose that 
	\begin{itemize}
		\item [(1)] $M$ has diameter growth of order $o(r^{\frac 1n})$, where $o(r^{\frac1n})$ is a function satisfying $\frac{o(r^\frac 1n)}{r^\frac 1n}\to0$ as $r\to\infty$,
		
		\item [(2)] the sectional curvature is bounded away from $-\infty$.
	\end{itemize}
	Then $M$ has finite topological type.
\end{theorem}

Note that both conditions (1) and (2) of Theorem \ref{AGFiniteTpType} are necessary. This can be seen from the examples constructed by Menguy (\cite{Men00,Men00a}), where he showed the existence of open manifolds with infinite topological type and positive Ricci curvature that satisfy either bounded diameter growth or sectional curvature bounded from below.

Starting from \cite{AG}, several results on finite topological type have been published (e.g. \cite{SW93,Shen96}). In general, these results follow a similar approach. First, a growth condition on geometric quantities (e.g. Theorem \ref{AGFiniteTpType} (1)) is required to ensure a small excess estimate, which is obtained using Abresch-Gromoll's excess estimate (Theorem \ref{ExcessEstimateAG} below). Additionally, a regularity assumption on metrics (e.g. Theorem \ref{AGFiniteTpType} (2)) is needed to guarantee a triangle comparison of Toponogov type. By combining the small excess estimate and the triangle comparison, it can be concluded that there are no critical points of distance functions, in the sense of Grove-Shiohama, outside a compact subset. Then the conclusion of finite topological type is derived from Grove-Shiohama's critical point theory (\cite{Ch91}).

However, the crucial tool in the common approach mentioned above, Grove-Shiohama's critical point theory, heavily relies on a triangle comparison of Toponogov type. As a result, it appears that the previous results inevitably required assumptions on the lower bound of sectional curvature or conjugate radius (\cite{DW95}) which serve as additional regularity assumptions on metrics. Consequently, the main focus of the past research was on finding suitable growth conditions, with less progress made in exploring weaker regularity assumptions.

Thanks to the recent advancements in Cheeger-Colding theory, the limitations previously imposed by sectional curvature or conjugate radius in investigating finite topological type theorems can now be overcome. In this paper, we diverge from previous studies and delve into finite topological type theorems for open manifolds with non-negative Ricci curvature, under the regularity assumption on metrics: almost maximal local rewinding volume (see condition (\ref{AlmostMaximalRewindingVolume}) below). It should be noted that our assumption does not guarantee the triangle comparison property, but it aligns better with the focus on studying Ricci curvature. Notably, Grove-Shiohama's critical point theory is not applicable to our specific case (Remark \ref{RemForMainThm} (2)). Instead, we utilize a technique that involves the construction of a smooth function by gluing Cheeger-Colding's almost splitting functions and approximating the distance function (\cite{CC96}). We then employ the transformation theorem (\cite{CJN21}) for almost splitting functions to establish the non-degeneracy of this function. It is worth noting that the approach we adopt is applicable even in cases where the sectional curvature or conjugate radius is bounded from below, as demonstrated by the work of \cite{HH22} (Remark \ref{RemarkOnNonDeg}). However, it should be noted that the concept of almost maximal local rewinding volume does not encompass these scenarios. In fact, our proof reveals that the only additional essential regularity assumption required is the monotonicity of the numbers of almost Euclidean factors, in the Gromov-Hausdorff sense, around a fixed point during the process of blowing up. This property is referred to as the generalized Reifenberg property in \cite[Definition 1.4]{HH22}.

Before stating our main theorem, we introduce some definitions. From now on, $(M,p)$ is always a pointed open  $n$-manifold with non-negative Ricci curvature and $r_p(x)=d(x,p)$ denotes the distance function to $p$. Firstly we recall the following definitions,

\begin{definition}[Local rewinding volume]
	The rewinding volume for an $r$-ball $B_r(x)$ is defined as, $$\widetilde{\vol}(B_r(x)):=\vol(B_r(\tilde x)),$$where $(\widetilde {B_r(x)},\tilde x)\to(B_r(p),x)$ is the universal cover.
\end{definition}

We say that an $n$-manifold $M$ satisfies uniformly $(\delta,\rho)$-rewinding Reifenberg, if there exist $\delta,\rho>0$, such that for any $x\in M$, $\widetilde{\vol}(B_{\rho}(x))\ge (1-\delta)w_n \rho^n$, where $w_n$ is the volume of unit ball, $B_1(0^n)$, in the $n$-Euclidean space. According to (\cite{CFG}), a complete $n$-manifold with bounded sectional curvature, $|\sec|\le1$, is uniformly $(\delta(n),\rho(n))$-rewinding Reifenberg, where $\delta(n),\rho(n)$ are universal small positive constants only depending on $n$.

One motivation for studying manifolds with specific types of rewinding volume conditions is their relevance to the research on topological properties of manifolds with bounded sectional curvature. The study of such manifolds has been greatly enriched by the collapsing theory developed by Cheeger, Fukaya, and Gromov (\cite{CFG}). As discussed in the previous paragraph, manifolds with bounded sectional curvature form a subset of the larger class of manifolds with lower Ricci curvature bound and almost maximal local rewinding volume. Rong (\cite{Ro18}) conjectures that manifolds with lower Ricci curvature bound and certain rewinding volume assumptions exhibit similar collapsing behavior and therefore share many geometric and topological properties with manifolds with bounded sectional curvature. This conjecture has been supported by a growing body of recent research (e.g. \cite{CRX1604,HKRX,Hua,Ro}).

In the study of finite topological type of open manifolds, the excess function plays a crucial role. As previously mentioned, different types of geometric growth conditions aim to ensure a small excess property. However, the definitions of the excess function for open manifolds may vary slightly in different literature. For the purpose of this study, we will adopt the following definition.
\begin{definition}[Excess function]
	Define $b^r_p(x)=r-d(x,\partial B_r(p))$, and $e_p(x)=r_p(x)-b_p^{2r_p(x)}(x)$,  where $\partial B_r(p)$ denotes the boundary of an $r$-ball around $p$. $e_p$ is called the excess function with respect to $p$.
\end{definition}

\begin{remark}
	Some different definitions are adopted in other places, for example, $E_p(x):=r_p(x)-\lim_{r\to\infty}b_p^{r}(x)$, and $E_p^\gamma(x):=r_p(x)-
	\lim_{t\to\infty}t-d(x,\gamma(t))$, $\bar E_p(x):=\inf_{\gamma}
	E_p^\gamma(x)$, where the infimum is taken over all ray $\gamma$ with unit speed form $p$. It's easy to see the relations $$e_p\le E_p\le \bar E_p\le E_p^\gamma.$$
\end{remark}

Now we are ready to state our main result.

\begin{theorem}\label{FiniteToplgType-Main}
	There exists $\delta(n)>0$ such that the following holds. Let $M$ be an open $n$-manifold with non-negative Ricci curvature. Suppose that there exist constants $\alpha\in[0,1]$, $s>0$ (we require $s\le\delta(n)$ when $\alpha=1$), satisfying, for any $x\in M$ outside a compact subset,  
	\begin{equation}\label{ExcessGrowthCondition}
		\frac{e_p(x)}{r_p(x)^\alpha}\le \delta(n)s, 
	\end{equation}

	\begin{equation}\label{AlmostMaximalRewindingVolume}
		\widetilde{\vol}(B_{s r_p(x)^\alpha}(x))\ge (1-\delta(n)) w_ns^n r_p(x)^{n\alpha}.
	\end{equation}
	 Then $M$ has finite topological type. 

\end{theorem}

\begin{remark}\label{RemForMainThm}\quad
	
	\begin{itemize}
		\item [(1)] If one replaces the regularity condition (\ref{AlmostMaximalRewindingVolume}) by asymptotically non-negative sectional curvature,
		\begin{equation}\label{SectionalCurvatureDecay}
			K(x)\ge-\left(\frac{C}{1+sr_p(x)^{\alpha}}\right)^{2},
		\end{equation}
		the conclusion of Theorem \ref{FiniteToplgType-Main} is still true, which has been proved in \cite{ShSh}. The approach to Theorem \ref{FiniteToplgType-Main} also applies to this case, although either of conditions (\ref{AlmostMaximalRewindingVolume}) and (\ref{SectionalCurvatureDecay}) is not included in each other. Similarly one may replace (\ref{AlmostMaximalRewindingVolume}) by conjugate radius $r_c(x)\ge sr_p(x)^\alpha$.
		
		\item [(2)] The exclusion of critical points for a distance function in the Grove-Shiohama sense outside a compact set depends on estimating the angles of a thin triangle. However, the behavior of angles may not be well-behaved for manifolds with non-negative Ricci curvature, even at regular points, without assuming anything about sectional curvature (as shown by examples in \cite{CN13}). This implies that the distance function $r_p$ in Theorem \ref{FiniteToplgType-Main} may have critical points that tend to infinity, which prevents the application of Grove-Shiohama's critical point theory.
	\end{itemize}

\end{remark}

Based on previous studies, the growth condition of the excess function (\ref{ExcessGrowthCondition}) can be inferred from the growth conditions of various geometric quantities such as diameter or volume. These growth conditions are summarized in the following proposition, and readers are referred to \cite{AG,SW93,Shen96,Xia02,Wa04,JY}, etc., for more details.

\begin{proposition}\label{FiniteToplgType-VariousGrowthCondition}
	Let $M$ be an open $n$-manifold with non-negative Ricci curvature. If $M$ satisfies one of the following growth conditions, for some $\alpha\in[0,1]$, small $\delta>0$, and all large $r$,
	\begin{itemize}
		\item [(1)] essential diameter of ends $\D_p(r)\le\delta^{\frac{n-1}{n}} r^\frac{1+(n-1)\alpha}{n}$,
		
		\item  [(2)] $\frac{\vol(B_r(p))}{v_p(r)}\le\delta^{\frac{n-1}{n}} r^{1+\frac{1+(n-1)\alpha}{n}}$, where $v_p(r)=\inf_{x\in\partial B_r(p)}\vol(B_1(x))$,
		
		\item  [(3)] for some $\nu>0$, $r^{\frac{(n-1)^2(1-\alpha)}{n}}\left(\frac{\vol B_r(p)}{r^n}-\nu\right)\le\delta^{\frac{(n-1)^2}{n}}$,
	\end{itemize}
	
	then outside a compact set, $$\frac{e_p(x)}{r_p(x)^\alpha}\le C\delta,$$where $C$ is a positive universal constant depending on $n$ for (1), (2), and $n,\nu$ for (3).

\end{proposition}

Proposition \ref{FiniteToplgType-VariousGrowthCondition} (2) and (3) have been established in previous studies. Proposition \ref{FiniteToplgType-VariousGrowthCondition} (2) is a corollary of Proposition \ref{FiniteToplgType-VariousGrowthCondition} (1), according to \cite{SW93}, while Proposition \ref{FiniteToplgType-VariousGrowthCondition} (3) has been proven in \cite{Shen96,OSY}. To facilitate readers, we provide detailed proofs of Proposition \ref{FiniteToplgType-VariousGrowthCondition} (2) and (3) at the end of subsection 2.1. However, Proposition \ref{FiniteToplgType-VariousGrowthCondition} (1) requires additional work as it does not assume sectional curvature. We will demonstrate that if an open manifold with non-negative Ricci curvature exhibits small linear growth of essential diameter of ends, then its essential diameter of ends and diameter of geodesic balls are nearly equal.

\begin{lemma}\label{SmallEsentialDiamterImpliesSmallExcess_Copy}
	For any $n,\epsilon>0$, there exists $\delta(n,\epsilon)>0$, satisfying the following properties. Let $(M,p)$ be an open $n$-manifold with non-negative Ricci curvature. If
	\begin{equation*}\label{SmallEsentialDiamter}
		\limsup_{r\to\infty}\frac{\D_p(r)}{r}\le\delta(n,\epsilon),
	\end{equation*}
	then $M$ is isometric to $\R\times N$ where $N$ is compact, or 
	\begin{equation*}\label{DiameterRatio}
		\limsup_{r\to\infty}\frac{\diam (\partial B_r(p))}{\D_p(r)}\le 1+\epsilon.
	\end{equation*}
	
\end{lemma}

The concept of the essential diameter of ends, $\D_p(r)$, as defined in Definition \ref{DefEssentialDiameter}, is valuable because it allows for estimating $\D_p(r)$ instead of $\diam(\partial B_r(p))$ by volume, which can be done in some cases (\cite{SW93}). The proof of the lemma mentioned above can be found in Lemma \ref{SmallEsentialDiamterImpliesSmallExcess} below. Now, Proposition \ref{FiniteToplgType-VariousGrowthCondition} (1) is a direct consequence of the estimate provided by Abresch-Gromoll (Theorem \ref{ExcessEstimateAG} below).

By substituting either one condition of Proposition \ref{FiniteToplgType-VariousGrowthCondition} for (\ref{ExcessGrowthCondition}) in Theorem \ref{FiniteToplgType-Main} and choosing a specific value of $\alpha\in[0,1]$, we can obtain a number of finite topological type theorems. Some interesting cases are listed below.

\begin{corollary}\label{Corrollaries-FinitTopType}
	Let $M$ be an open $n$-manifold with non-negative Ricci curvature.
	\begin{enumerate}
		\item If $M$ has Euclidean volume growth, and any infinity cone of $M$ is smooth,
		
		\item If $M$ satisfies uniformly rewinding $(\delta(n),\rho)$-Reifenberg, and either of $\D_p(r)=o(r^\frac{1}{n})$, or $\frac{\vol(B_r(p))}{v_p(r)}=o (r^{1+\frac1n})$, 
		
		\item If $n=4$, and the supremum of Ricci curvature $\sup\Ric<\infty$, and $M$ satisfies $\sup_{r>0}\D_p(r)<\infty$ or $\frac{\vol (B_r(p))}{r^4}=\nu+o(r^{-\frac94})$, for some $\nu>0$.

	\end{enumerate}

	Then, $M$ has finite topological type.
\end{corollary}

Some remarks about Corollary \ref{Corrollaries-FinitTopType} are listed.

\begin{remark}\quad

	(1) It appears that the assumption regarding the infinity cone in Corollary \ref{Corrollaries-FinitTopType} (1) is quite restrictive. However, based on examples from \cite{Men00a}, we cannot remove this additional assumption. One important example that satisfies the conditions of Corollary \ref{Corrollaries-FinitTopType} (1) is the class of Ricci flat open $4$-manifolds with Euclidean volume growth, as shown by Cheeger-Naber's codimensional $4$ theorem (\cite{CN}).
	
	Indeed, Colding-Minicozzi (\cite{CoMi12}) proved that if a Ricci flat open manifold with Euclidean volume growth has one smooth infinity cone, then it has a unique infinity cone. Their proof also indicates such a manifold has finite topological type.

	(2)	In Corollary \ref{Corrollaries-FinitTopType} (2), if we replace the rewinding $(\delta(n),\rho)$-Reifenberg condition by $K_M\ge -C>-\infty$, then the case of $\D_p(r)=o(r^{\frac{1}{n}})$ recovers Theorem \ref{AGFiniteTpType} and $\frac{\vol(B_r(p))}{v_p(r)}=o (r^{1+\frac1n})$ is the case dealt with by \cite{SW93}.
		
	(3)	In \cite{Men00}, Menguy constructed a complete $4$-manifold with positive Ricci curvature and bounded diameter growth which has infinite topological type. Corollary \ref{Corrollaries-FinitTopType} (3) demonstrate that these examples must satisfy $\sup\Ric=+\infty$.

\end{remark}

We arrange this paper as follows. In section 2, we first introduce the relations between growth of excess function and various geometric quantities, which are applied to prove Proposition \ref{FiniteToplgType-VariousGrowthCondition}, then we recall Cheeger-Colding's almost splitting functions and prove that there exits an almost splitting function on a large annulus. Section 3 is devoted to glue the almost splitting functions on annuli by a partition of unity. In Section 4, we apply a transformation theorem for almost splitting maps to show the function constructed in Section 3 has no critical points outside a compact set, hence which finishes the proof of Theorem \ref{FiniteToplgType-Main}, and then we finish the proof of Corollary \ref{Corrollaries-FinitTopType}. In appendix, we give the proofs of Lemma \ref{LargeVolumeGrowth}, \ref{SmallVolumeGrowth} and a sketched proof of Theorem \ref{TransThm} for readers' convenience.

\section{Preliminary}

In this paper, we use $C(c_1,...,c_s)$ to denote some universal positive constants, only depending on parameters $c_1,...,c_s$. And $\Psi(x_1,..,x_k|c_1,...,c_s)$ to denote some universal positive functions satisfying $\Psi\to0$ as $x_1,...,x_k\to0$ when $c_1,...,c_s$ are fixed. Note that their specific values may change from line to line without specification if there is no ambiguity.

\subsection{Relations between Various Growth Conditions}

As mentioned in the introduction, among existing works, one indispensable tool in proving a finite topological type theorem of an open manifolds with nonnegative Ricci curvature is the Abresch-Gromoll's excess estimate, which allows that one estimates some distance relations in thin triangles only assuming non-negative Ricci curvature.

\begin{theorem}\cite{AG}\label{ExcessEstimateAG}
	Let $M$ be an open $n$-manifold with non-negative Ricci curvature. Let $\gamma$ be a minimal geodesic joining $p,q\in M$. Then for any $x\in M$, $$e_{p,q}(x)=d(x,p)+d(x,q)-d(p,q)\le8\left(\frac{h^n}{s}\right)^{\frac{1}{n-1}},$$where $h=d(x,\gamma)$, $s=\min\{d(x,p),d(x,q)\}$.
\end{theorem}
A slight generalization of the Abresch-Gromoll's excess estimate refers to \cite{CC96}, where they use excess estimate to establish the quantitative splitting theorem (Theorem \ref{Quantitativesplitting} below).

In the study of an open manifold with non-negative Ricci curvature, in order to guarantee the above excess estimate, a common idea is to estimate the growth of the ray density function $\ray_p(r):=\sup_{x\in \partial B_r(p)}d(x,R_p\cap\partial B_r(p))$, where $R_p$ is the union of all rays starting from $p$, by certain diameter or volume growth conditions. In fact, according to Theorem \ref{ExcessEstimateAG}, one immediately has, 
\begin{equation}\label{RayDensityControlExcess}
e_p(x)\le 8\left(\frac{\ray_p(r_p(x))^n}{r_p(x)}\right)^{\frac{1}{n-1}}.
\end{equation}

Now we recall some effective geometric quantities growth conditions including diameter and volume, which are used to control the growth of the ray density function.

First we recall the properties of Euclidean volume growth. 

\begin{definition}[Euclidean volume growth]
	For $\nu>0$, we call an open $n$-manifold $M$ has $\nu$-Euclidean volume growth, if there exists $p\in M$ s.t.  
	\begin{equation*}\label{EuclideanVolumeGrowth}
		\liminf_{r\to\infty}\frac{\vol(B_r(p))}{r^n}=\nu.
	\end{equation*}
	
\end{definition}
By relative volume comparison, if $M$ has non-negative Ricci curvature, then the above limit always exists and does not depend on $p$. For open manifolds with non-negative Ricci curvature and Euclidean volume growth, the following theorem is fundamental.
\begin{theorem}(\cite{CC96})\label{Thm-VolumeCone}
	Let $M$ be an open manifold with non-negative Ricci curvature. If $M$ has Euclidean volume growth, then each infinity cone of $M$ is a metric cone.
\end{theorem}

The following lemma tells that one may use the rate of Euclidean volume growth to control ray density function. For reader's convenience, we give its proof in appendix.

\begin{lemma}[Large volume growth, \cite{Shen96,OSY}]\label{LargeVolumeGrowth}
	Let $(M,p)$ be an open $n$-manifold with non-negative Ricci curvature and $\nu$-Euclidean volume growth. Then for large $r$, $$\ray_p(r)\le C(n) (\nu^{-1})^{\frac{1}{n-1}}\left(\frac{\vol(B_r(p))}{r^n}-\nu\right)^{\frac1{n-1}}r.$$
\end{lemma}

Next we discuss the diameter case. The most natural diameter growth function is $\diam (\partial B_r(p))$, $r>0$. Obvious that $\ray_p(r)\le \diam (\partial B_r(p))$. However, sometimes, it's not easy to control this extrinsic diameter by volume. In \cite{AG}, Abresch-Gromoll introduced a modified diameter growth function, $\D^{AG}_p(r)$, and in \cite{Ch91}, Cheeger introduced a so-called essential diameter growth function, $\D^{Ch}_p(r)$. In this paper, we adopt the following definition introduced in \cite{SW93}.

\begin{definition}[Essential diameter of ends]\label{DefEssentialDiameter}
	Let $C(p,r)$ be the union of unbounded components of $M\backslash \overline{B_{r}(p)}$. Put
	$$\D_p (r)=\sup\diam(\Sigma_r),$$where the superemum is taken over all connected components, $\Sigma_r$, of boundaries $\partial C(p,r)$, which satisfies $\Sigma_r\cap R_p\neq\emptyset$, and the diameter of $\Sigma_r$ is measured with respect to extrinsic distance of $M$. 
\end{definition}
If one drops the restriction that $\Sigma_r\cap R_p\neq\emptyset$, then the above definition gives $\D^{Ch}_p(r)$. In general, it is obvious that $\D_p(r)\le\D^{Ch}_p(r)\le \diam(\partial B_r(p))$ and $\D^{Ch}_p(r)\le \D^{AG}_p(r)$. In the case of non-negative Ricci curvature, \cite[Proposition 4.3]{AG} asserts that, the boundary of each component of $M\backslash B_r(p)$ is connected, consequently, $\D_p(r)=\D^{Ch}_p(r)$. 

The next lemma, whose proof we put in appendix for readers' convenience, states that one may use small volume growth to control essential diameter growth.

\begin{lemma}[Small volume growth, \cite{SW93}]\label{SmallVolumeGrowth}
	Let $(M,p)$ be an open $n$-manifold with non-negative Ricci curvature. Then $\D_p (r)\le\frac{4}{v_p(r)}\vol(A_{r-2,r+2}(p))$.
\end{lemma}

In the case of small linear growth of $\D_p(r)$, $\D_p(r)$ is almost the same as $\diam (B_r(p))$. That is the Lemma \ref{SmallEsentialDiamterImpliesSmallExcess_Copy}. For convenience of readers, we restate it below.

\begin{lemma}\label{SmallEsentialDiamterImpliesSmallExcess}
	For any $n,\epsilon>0$, there exists $\delta(n,\epsilon)>0$, satisfying the following properties. Let $(M,p)$ be an open $n$-manifold with non-negative Ricci curvature. If
	\begin{equation}\label{SmallEsentialDiamter}
		\limsup_{r\to\infty}\frac{\D_p(r)}{r}\le\delta(n,\epsilon),
	\end{equation}
	then $M$ is isometric to $\R\times N$ where $N$ is compact, or 
	\begin{equation}\label{DiameterRatio}
		\limsup_{r\to\infty}\frac{\diam (\partial B_r(p))}{\D_p(r)}\le 1+\epsilon.
	\end{equation}
	
\end{lemma}
\begin{proof}
	By Cheeger-Gromoll's splitting theorem, $(M,p)$ is isometric to $(\R^k\times N,(0^k,p_N))$ where $N$ contains no lines.
	
	Case 1: If $k\ge2$, let $u,v$ be orthonormal vectors in $\R^k$. Then for each $r>0$, $((u\cos s+v\sin s)r,p_N)$, $s\in[0,\frac\pi2]$, is a shortest geodesic in $\partial B_r(p)$ joining $(ur,p_N)$ and $(vr,p_N)$. Hence $\D_p(r)\ge\frac\pi2 r$ which is a contradiction to (\ref{SmallEsentialDiamter}) provided $\delta(n,\epsilon)<\frac12$.

	Case 2: If $k=1$ and $N$ is not compact. Let $\gamma$ be a ray emanating from $p$ with unit speed which is parallel to the $N$-factor. Then $\gamma$ and the $\R$-factor span a half plane which will yield a contradiction as in case 1.
	
	Case 3: The remain is to show if $M$ contains no lines, then (\ref{DiameterRatio}) holds. Let $\gamma$ be a ray emanating from $p$. Assume there exists $x\in\partial B_{r}(p)$. Join a minimal geodesic $\omega$ with unit speed from $x$ to $\gamma(r)$. Let $$\rho:=d(\omega(s),p):=\min_{t\in[0,d(x,\gamma(r))]}\{d(\omega(t),p)\}.$$ 
	Note that $\rho\le r$.

	For a ray $\gamma$, let $\Sigma_\gamma(r)$ be the connected component of $\partial C(p,r)$ containing $\gamma(r)$. Join a shortest geodesic $\alpha$ with unit speed from $p$ to $x$. According to \cite[Proposition 4.3]{AG} (see also \cite{Sor01}), the boundary of the component of the complement $M\backslash B_{\rho}(p)$ intersecting $\gamma$ must be connected. This shows $\omega(s),\alpha(\rho)\in\Sigma_{\gamma}(\rho)$. Letting $D(\rho)=\max_{y\in \Sigma_{\gamma}(\rho)}\{d(y,\gamma(\rho))\}$, so $$\max\{d(\omega(s),\gamma(\rho)),d(\alpha(\rho),\gamma(\rho))\}\le D(\rho).$$
	
	Then,
	\begin{equation}\label{r-rho}
		d(x,\gamma(r))\le2(r-\rho)+d(\alpha(\rho),\gamma(\rho))\le 2(r-\rho)+D(\rho).
	\end{equation}

	Now put $(\bar M,\bar q,\bar d)=(M,\omega(s),\frac{1}{D(\rho)}d)$. Since $M$ contains no lines, $\rho\to\infty$ as $r\to\infty$; otherwise, a contradicting argument shows that there exists a sequence of $\omega_i$ sub-converges to a line. So for $r$ large enough, we may assume $\rho$ is sufficiently large s.t., $\frac{\D_p(\rho)}{\rho}\le 2\delta$ by (\ref{SmallEsentialDiamter}). Now we have, $$\bar d(\bar q,p)=\frac{\rho}{D(\rho)}\ge (2\delta)^{-1},\quad \bar d(\bar q,\gamma(2\rho))\ge \frac{\rho}{D(\rho)}\ge (2\delta)^{-1},$$
	$$\bar d(\bar q,\gamma(\rho))\le1.$$ By Sublemma \ref*{NonBranchLemma}, 
	\begin{equation}\label{Estimate(r-rho)}
		r-\rho\le\Psi(\delta|n)D(\rho).
	\end{equation}
	Combining inequality (\ref{r-rho}) and (\ref{Estimate(r-rho)}) yields 
	\begin{equation}\label{D(rho)}
		d(x,\gamma(r))\le (1+\Psi(\delta|n))D(\rho).
	\end{equation}
	
	For $t>>2r$, let $y\in\Sigma_{\gamma}(\rho)$ s.t. $d(y,\gamma(\rho))=D(\rho)$. Let $y_t\in \overline {y,\gamma(t)}\cap \partial B_{r}(p)$, where $\overline {y,\gamma(t)}$ is a shortest geodesic from $y$ to $\gamma(t)$. When $r$ is large, $D(\rho)\le\D_p(\rho)\le2\delta\rho$, so by excess estimate,
	$$d(y,p)+d(y,\gamma(t))-d(p,\gamma(t))\le8\left(\frac{D(\rho)^n}{\rho}\right)^{\frac{1}{n-1}}\le 8(2\delta)^{\frac{1}{n-1}}D(\rho),$$  which yields, 
	\begin{align*}
		&d(y,y_t)\\=&d(y,\gamma(t))+d(y,p)-d(p,\gamma(t))-(d(y_t,\gamma(t))+d(y_t,p)-d(p,\gamma(t)))+d(y_t,p)-d(y,p)\\\le&2(d(y,\gamma(t))+d(y,p)-d(p,\gamma(t)))+r-\rho\\\le& r-\rho+16(2\delta)^{\frac{1}{n-1}}D(\rho)\le \Psi(\delta|n)D(\rho).
	\end{align*}
	Then
	$$d(y_t,\gamma(r))\ge d(y,\gamma(\rho))- d( y,y_t)- d(\gamma(\rho),\gamma(r))\ge (1-\Psi(\delta|n))D(\rho),$$ so
	\begin{equation}\label{EstimateOfD_p(r)}
		\D_p(r)\ge d(y_t,\gamma(r))\ge(1-\Psi(\delta|n))D(\rho).
	\end{equation}
	
	Now combining (\ref{D(rho)}) and (\ref{EstimateOfD_p(r)}) completes the proof.

	What left is to prove the following sublemma. For points $x,y\in M$, we use $\overline{x,y}$ to denote a shortest geodesic from $x$ to $y$.
	\begin{sublemma}\label{NonBranchLemma}
		Let $M$ be an open $n$-manifold with non-negative Ricci curvature. Suppose, for $x,y,p,q\in M$, $$d(x,p)\ge\delta^{-1},\quad d(x,q)\ge\delta^{-1},$$ $$d(x,\overline{p,q})\le1.$$If there exist $y\in M$, $z\in\overline{p,q}$, s.t. $d(p,y)=d(p,z)$, $x\in\overline{y,z}$, $d(p,x)\le d(p,\overline{y,z})$, then $d(p,z)-d(p,x)\le\Psi(\delta|n)$.
		
	\end{sublemma}

	It's by a contradicting argument; Assume $\{M_i\}$ is a sequence of open manifolds with non-negative Ricci curvature and for $x_i,y_i,z_i,p_i,q_i\in M_i$, $$d(x_i,p_i)\ge\delta_i^{-1},\quad d(x_i,q_i)\ge\delta_i^{-1},\quad \delta_i\to0,$$ $$d(x_i,\overline{p_i,q_i})\le1,$$ $$z_i\in\overline{p_i,q_i},\quad x_i\in\overline{y_i,z_i},$$ $$d(p_i,x_i)\le d(p_i,\overline{y_i,z_i}),$$ $$d(p_i,y_i)=d(p_i,z_i),$$ and $$d(p_i,z_i)-d(p_i,x_i)>\eta.$$
	
	Put $h_i(\cdot)=d(\cdot,p_i)-d(x_i,p_i)$.
	By continuity of $h_i$, there exists $z_i'\in\overline{x_i,z_i}$ s.t. $h_i(z_i')=\eta$. Note that $d(z_i',x_i)\ge \eta$.
	
	We claim there exists $R>0$, s.t. $d(x_i,z_i')\le R$. If not, up to a subsequence, $d(x_i,z_i)\ge d(x_i,z_i')\to\infty$. By excess estimate, $$ d(x_i,p_i)+d(x_i,z_i)-d(p_i,z_i)\to0.$$	
	Hence, $$0\le d(x_i,z_i')-\eta=d(x_i,p_i)+d(x_i,z_i')-d(p_i,z_i')\le d(x_i,p_i)+d(x_i,z_i)-d(p_i,z_i) \to0$$ which is a contradiction to the contradicting assumption $d(x_i,z_i')\to\infty$. Hence the claim that there exists $R>0$, s.t. $d(x_i,z_i')\le R$ follows.
	
	Now passing to a subsequence, according to quantitative splitting theorem \ref{Quantitativesplitting} below, we assume $(M_i,x_i,z_i')\GH(\R\times X,x_0,z_0)$ and $\overline{y_i,z_i}$ converges to a geodesic $\omega$ passing through $x_0$ and $z_0$ with length at least $\eta$. where $\GH$ means Gromov-Hausdorff convergence. And by the proof of Theorem \ref{Quantitativesplitting}, $h_i$ converges to $h_0:(\R\times X,x_0)\to(\R,0)$ the standard projection. Specially $h_0(z_0)=\lim_{i}h_i(z_i')=\eta>0$, $h_0(x_0)=0$ and $h_0(\omega(t))=\lim_ih_i(\overline{y_i,z_i}(t))\ge0$. By the product structure, we conclude that one end of $\omega$ is $x_0$. This implies the $d(x_i,y_i)\to0$ which is impossible since $d(x_i,y_i)\ge d(y_i,p_i)-d(x_i,p_i)=d(z_i,p_i)-d(x_i,p_i)>\eta$. Here the proof of the sublemma is complete.

\end{proof}

Now we are ready to prove Proposition \ref{FiniteToplgType-VariousGrowthCondition}:

(1) It follows directly by combining Lemma \ref{SmallEsentialDiamterImpliesSmallExcess} and inequality (\ref{RayDensityControlExcess}). 

(2) By Lemma \ref{SmallVolumeGrowth} and relative volume comparison, 
\begin{align*}
	\D_p(r)\le 4\frac{\vol(A_{r-2,r+2}(p))}{v_p(r)}&\le4\left((1+\frac2r)^n-(1-\frac2r)^n\right)\frac{\vol(B_r(p))}{v_p(r)}\\&\le \frac{C(n)}r \delta^{\frac{n-1}{n}} r^{1+\frac{1+(n-1)\alpha}{n}}=C(n)\delta^{\frac{n-1}{n}} r^{\frac{1+(n-1)\alpha}{n}},
\end{align*}
which is the case (1).

(3) By Lemma \ref{LargeVolumeGrowth}, $$\ray_p(r)\le C(n) (\nu^{-1})^{\frac{1}{n-1}}\left(r^{\frac{(n-1)^2(\alpha-1)}{n}}\delta^{\frac{(n-1)^2}{n}} \right)^{\frac1{n-1}}r=C(n,\nu)r^{1+\frac{(n-1)(\alpha-1)}{n}}\delta^{\frac{n-1}{n}}.$$

By inequality (\ref{RayDensityControlExcess}), for any $x\in\partial B_r(p)$,
$$e_p(x)\le 8\left(\frac{\ray_p(r)^n}{r}\right)^{\frac1{n-1}}\le C(n,\nu)r^{\alpha}\delta.$$

\subsection{Cheeger-Colding's Almost Splitting Functions}

First we recall the notion of $(\delta,k)$-splitting maps, which approximate coordinate functions in $H^2$-sense.
 
\begin{definition}[\cite{CC96}]\label{DefSplittingMap}
	A map $u=(u^1,...,u^k):B_r(p)\to\R^k$ is called a $(\delta,k)$-splitting map, if it satifies, for $\alpha,\beta=1,...,k$,
	\begin{itemize}

		\item [(1)] $\sup_{x\in B_r(p)}\|\grad u^\alpha(x)\|\le 1+\delta$,
		
		\item [(2)] $\dashint_{B_r(p)}|\spa{\grad u^{\alpha},\grad u^{\beta}}-\delta^{\alpha\beta}|\le \delta$,
		
		\item [(3)] $r^2\dashint_{B_r(p)}\|\hess u^\alpha\|^2\le \delta$,
	\end{itemize}
	where $\dashint_{B_r(p)}:=\frac{1}{\vol(B_r(p))}\int_{B_r(p)}$.

\end{definition}

If each $u^\alpha$ is harmonic, we call such $u$ a harmonic $(\delta,k)$-splitting map. And if $k=1$, we call $u$ a $\delta$-splitting function. One fundamental application of almost splitting functions is the following quantitative splitting theorem.

\begin{theorem}[\cite{CC96}]\label{Quantitativesplitting}
	Let $M$ be a complete $n$-manifold with $\Ric_M\ge-(n-1)\delta$. If $p_+,p_-,p\in M$ satisfy $$d(p,p_\pm)\ge\delta^{-1},\,d(p,p_+)+d(p,p_-)-d(p_+,p_-)\le\delta,$$ then there exists a harmonic $\Psi(\delta|n,R)$-splitting function $u:B_{20R}(p)\to\R$ which satisfies $|u-d(x,p_+)|_{C^0(B_{20R}(p))}\le\Psi(\delta|n,R)$. Furthermore, there exists a map $P:B_R(p)\to u^{-1}(u(p))$ s.t., the map $B_R(p)\to\R\times u^{-1}(u(p))$ defined by $x\mapsto(u(x),P(x))$ is a $\Psi(\delta|n,R)$-Gromov-Hausdorff approximation.
\end{theorem}

During the proof of Theorem \ref{Quantitativesplitting}, Abresch-Gromoll's excess estimate plays a key role to show the existence of a $\Psi(\delta|n,R)$-splitting function $u$.

In our application, we need to find almost splitting functions on a sequence of large annuli. Let $A_{r,R}(p):=r_p^{-1}(r,R)$.

\begin{lemma}\label{0108lem}
	Let $(M,p)$ be a complete $n$-manifold with $\Ric_M\ge0$. If for $L\ge 600R>0$, $e_p|_{A_{L-100R,L+100R}(p)}\le \epsilon$, then  the solution of the equation,
	$$\begin{cases}
		\lap h=0,\quad\mathrm{on}\;A_{L-100R,L+100R}(p),
		\\h=r_p,\quad\mathrm{on}\;\partial{A_{L-100R,L+100R}(p)},
	\end{cases}$$
	satisfies that, for any $x\in A_{L-10R,L+10R}(p)$,
	\begin{itemize}

		\item [(1)] $|h(x)-r_p(x)|\le\Psi(\epsilon,L^{-1}|n,R)$,
		
		\item[(2)] $\|\grad h(x)\|\le 1+\Psi(\epsilon,L^{-1}|n,R)$,
		
		\item[(3)] $\dashint_{B_{10R}(x)}|\|\grad h\|^2-1|\le\Psi(\epsilon,L^{-1}|n,R)$,
		
		\item[(4)] $\dashint_{B_{10R}(x)}\|\hess h\|^2\le\Psi(\epsilon,L^{-1}|n,R)$.
	\end{itemize}

\end{lemma}
\begin{proof}
	By the condition of $e_p$, according to Theorem \ref{Quantitativesplitting}, for any $x\in A_{L-10R,L+10R}(p)$, there exists a harmonic $\Psi(\epsilon,L^{-1}|n,R)$-splitting function $u_x:B_{20R}(x)\to\R$ satisfying $|u_x-r_p|\le\Psi(\epsilon,L^{-1}|n,R)$. Hence once we verify property (1) for $h$, then properties (2),(3) are implied by the fact $\|\grad h-\grad u_x\|_{C^0(B_{15R}(x))}\le\Psi(\epsilon,L^{-1}|n,R)$ (Cheng-Yau's gradient estimate), and (4) is implied by Bochner formula. 
	
	The proof follows the same lines of the one of Theorem \ref{Quantitativesplitting}. Let $h^+:=h$ and $h^-$ be solutions of the following equations respectively, 
	
	$$\begin{cases}
		\lap h^+=0,\quad\mathrm{on}\;A_{L-100R,L+100R}(p),
		\\h^+=r_p,\quad\mathrm{on}\;\partial{A_{L-100R,L+100R}(p)},
	\end{cases}$$
	$$\begin{cases}
		\lap h^-=0,\quad\mathrm{on}\;A_{L-100R,L+100R}(p),
		\\h^-=-b_p^{2L-200R},\quad\mathrm{on}\;\partial{A_{L-100R,L+100R}(p)}.
	\end{cases}$$

	By Laplacian comparison, $$\lap (h^+-r_p)\ge- \frac{n-1}{r_p}\ge- \frac{n-1}{L-100R}\ge -2(n-1)L^{-1},$$ on $A_{L-100R,L+100R}(p)$ in barrier sense. Let $$\underline L_R(r)=\frac{1}{2n}(r^2-R^2)+\frac{R^n}{n(n-2)}(r^{2-n}-R^{2-n})$$ be the function satisfying
	$$\begin{cases}
		\underline{L}_R''(r)+\dfrac{n-1}{r}\underline{L}_R'(r)=1,
		\\\underline{L}_R(R)=0,
		\\\underline{L}_R'(r)<0,\,r\in (0,R).
	\end{cases}$$
	
	For $x\in \overline{A_{L-100R,L+100R}(p)}$, put $r_\Sigma(x)=d(x,\partial B_{L-110R}(p))$ and $\phi_R(x)=\underline L_{210R}(r_\Sigma(x))$. Thus 
	\begin{align*}	
		\lap\phi_R(x)=&\underline{L}_{210R}''(r_\Sigma(x))+\lap r_\Sigma(x) \underline{L}_{210R}'(r_\Sigma(x))\\\ge&\underline{L}_{210R}''(r_\Sigma(x))+\frac{n-1}{r_\Sigma(x)} \underline{L}_{210R}'(r_\Sigma(x))=1.
	\end{align*}
	Thus $\lap (h^+-r_p+2(n-1)L^{-1}\phi_R)(x)\ge0$ for $x\in A_{L-100R,L+100R}(p)$. By maximum principle, $$h^+(x)-r_p(x)+2(n-1)L^{-1}\phi_R(x)\le 2(n-1)L^{-1}\phi_R(x_0),$$ for some $x_0\in \partial{A_{L-100R,L+100R}(p)}$. Noting that $r_\Sigma(x),r_\Sigma(x_0)\in [10R,210R]$. Thus 
	\begin{equation}\label{PlusEst}
		h^+(x)-r_p(x)\le4(n-1)L^{-1}\underline{L}_{210R}(10R)=C(n)R^2L^{-1}.
	\end{equation}

For the other side, by Laplacian comparison, $$\lap (h^-+b_p^{2L-200R})\ge- \frac{n-1}{d(\cdot,\partial B_{2L-200R}(p))}\ge- \frac{n-1}{L-300R}\ge -2(n-1)L^{-1},$$ on $A_{L-100R,L+100R}(p)$ in barrier sense. Repeating the same procedure as above, we also have for any $x\in A_{L-100R,L+100R}(p)$, 
\begin{equation}\label{MinusEst}
	h^-(x)+b_p^{2L-200R}(x)\le C(n)R^2L^{-1}.
\end{equation}

Noting that $(h^++h^-)|_{\partial A_{L-100R,L+100R}(p)}=r_p-b_p^{2L-200R}\in[0,\epsilon]$ and $\lap (h^++h^-)=0$, again by maximum principle, for any $x\in A_{L-100R,L+100R}(p)$, $h^+(x)+h^-(x)\in [0,\epsilon]$. Hence $$h^+(x)-r_p(x)\ge -h^-(x)-b^{2L-200R}_p(x)-\epsilon\stackrel{(\ref{MinusEst})}\ge-C(n)R^2L^{-1}-\epsilon,$$ combining (\ref{PlusEst}), which proves (1). 
	
\end{proof}

\section{Approximate Distance Function}

The following result states that if $M$ has a small excess, then there exists a smooth function $h$ approximating $r_p$. $h$ will be our Morse function in the proof of Theorem \ref{FiniteToplgType-Main}.

\begin{theorem}\label{MorseFunctionExistence}
	Let $(M,p)$ be an open $n$-manifold with non-negative Ricci curvature. If there exists $\alpha\in[0,1]$, for any $x\in M$ outside a compact subset, 
	\begin{equation*}
		\frac{e_p(x)}{ r_p(x)^\alpha}\le \delta s
	\end{equation*}
	then there exists $R>0$, and a smooth function $h:M\backslash B_R(p)\to\R$ satisfying that for any $x\in M\backslash B_R(p)$,
	
	\begin{itemize}
		\item [(1)] $\frac{|h(x)-r_p(x)|}{s r_p(x)^{\alpha}}\le\Psi(s R^{\alpha-1},\delta|n)$,
		
		\item [(2)] Restriction $h:B_{s r_p(x)^\alpha}(x)\to\R$ is a $\Psi(s R^{\alpha-1},\delta|n)$-splitting function.
		
		\item [(3)] $s r_p(x)^\alpha|\lap h(x)|\le \Psi(s R^{\alpha-1},\delta|n)$. 
		
	\end{itemize}
	
\end{theorem}

The construction of $h$ in Theorem \ref{MorseFunctionExistence} is from gluing harmonic almost splitting functions on annuli by partition of unity. The following lemma is an adaption of Lemma \ref{0108lem}.

\begin{lemma}\label{LocalFlow}
Let $(M,p)$ be an open $n$-manifold with non-negative Ricci curvature. Assume that there exists $\alpha\in[0,1]$, for any $x\in M$ outside a compact subset, $$\frac{e_p(x)}{r_p(x)^\alpha}\le \delta s.$$ Then there exists $R_0>0$, s.t. for any $R>R_0$, there exists a harmonic function, $$h_R:A_{R-20s R^\alpha,R+20s R^\alpha}(p)\to\R$$ satisfying that, for any $x\in A_{R-10s  R^\alpha,R+10s R^\alpha}(p)$,
\begin{itemize}
	\item[(1)] $|h_R(x)-r_{p}(x)|\le s R^\alpha\Psi(s R^{\alpha-1},\delta|n)$, 
	
	\item[(2)] Restricting on each $B_{10s R^\alpha}(x)$, $h_R|_{B_{10s R^\alpha}(x)}$ is a $\Psi(s R^{\alpha-1},\delta|n)$-splitting function,

\end{itemize}

\end{lemma}
\begin{proof}

	Let $(\bar M,\bar p)=((sR^\alpha )^{-1}M,p)$. Then for any $\bar x\in \bar A_{s^{-1}R^{1-\alpha}-100,s^{-1}R^{1-\alpha}+100}(\bar p)$, $$e_{\bar p}(\bar x)=\frac{e_p(\bar x)}{sR^\alpha }\le \delta\left(\frac{r_p(\bar x)}{R}\right)^\alpha\le2\delta,$$provided $sR^{\alpha-1}<0.01$, where $\bar A$ denotes a metric annulus on $\bar M$. Applying Lemma \ref{0108lem}, there exists a harmonic $$\bar h_R:\bar A_{s^{-1}R^{1-\alpha}-100,s^{-1}R^{1-\alpha}+100}(\bar p)\to\R$$ satisfying, for any $\bar x\in \bar A_{s^{-1}R^{1-\alpha}-10,s^{-1}R^{1-\alpha}+10}(\bar p)$,
	\begin{itemize}

		\item [(1)] $|\bar h_R(\bar x)-r_{\bar p}(\bar x)|\le\Psi$,
		
		\item[(2)] $\|\bar \grad \bar h_R(\bar x)\|\le 1+\Psi$,
		
		\item[(3)] $\dashint_{B_{10}(\bar x)}|\|\bar \grad \bar h_R\|^2-1|\le\Psi$,
		
		\item[(4)] $\dashint_{B_{10}(\bar x)}\|\overline\hess \bar h_R\|^2\le\Psi$,
	\end{itemize}
	where $\Psi=\Psi(s R^{\alpha-1},\delta|n)$, and $\overline\grad$($\overline\hess$) means gradient(Hessian) respect to the metric on $\bar M$. Then $h_R:=s R^\alpha \bar h_R$ is what we need.
	
\end{proof}

In the following discussion, we will make $s R^{\alpha-1},\delta,$ become smaller and smaller, only depending on $n$, from line to line without specification.

Let $R=R_0$ as in Lemma \ref{LocalFlow}. Define $r_0=2R$, $r_{i+1}=r_i+4s r_i^\alpha$. We have the following elementary properties about $r_i$. 
\begin{proposition}\label{GoodCoverProperties}\quad

\begin{itemize}
	\item [(1)] For any $x\in M\backslash B_{r_0}(p)$, there exists $i\ge0$, s.t. $x\in A_{r_i-2.5s r_i^\alpha,r_i+2.5s r_i^\alpha}(p)$,
	
	\item [(2)] If $s r_0^{-1}$ is sufficiently small, then for any fixed $x\in M$, number of elements in $\{i=0,1,..|x\in A_{r_i-10s r_i^\alpha,r_i+10s r_i^\alpha}(p)\}$ is not larger than $6$. 
\end{itemize}\label{YYY}
	
\end{proposition}
\begin{proof}

	(1) The conclusion follows by the fact that for each $i\ge0$, $r_i+2.5s r_i^\alpha>r_{i+1}-2.5s r_{i+1}^\alpha$. 
	
	(2) 
	Note that for any $i\ge0$, 
	\begin{equation}\label{RatioEstOfr_i}
		1\le \frac{r_{i+1}}{r_{i}}=1+4\frac{s}{r_i^{1-\alpha}}\le 1+4\frac{s}{r_0^{1-\alpha}}\le 1+4s R^{\alpha-1}.
	\end{equation}
	
	Hence if $s R^{\alpha-1}<0.1$, then $r_{i}-10s r_{i}^\alpha\le r_{i+1}-10s r_{i+1}^\alpha$ holds for any $i\ge0$. It suffices to show  $7\notin\{k=1,2,...|r_{i+k}-10s r_{i+k}^\alpha\le r_i+10s r_i^\alpha\}$ for any $i\ge0$. Otherwise, we assume, there exists $i\ge0$ s.t., $$r_{i+7}-10s r_{i+7}^\alpha\le r_i+10s r_i^\alpha.$$ Then
	$$r_{i+7}-r_i=4s\left(r_{i+6}^\alpha+...+r_i^\alpha\right)\le 10s(r_i^\alpha+r_{i+7}^\alpha),$$
	\begin{equation}\label{Estm_Sum}
		\left(\frac{r_{i+6}}{r_{i+7}}\right)^\alpha+...+\left(\frac{r_{i+1}}{r_{i+7}}\right)^\alpha\le 2.5+1.5\left(\frac{r_{i}}{r_{i+7}}\right)^\alpha\le 4.
	\end{equation}

Plugging (\ref{RatioEstOfr_i}) into the left hand side of (\ref{Estm_Sum}), yields

$$\frac{6}{(1+4sR^{\alpha-1})^{6\alpha}}\le\frac{1}{(1+4sR^{\alpha-1})^\alpha}+...+\frac{1}{(1+4sR^{\alpha-1})^{6\alpha}}\le4.$$ So$$\frac{1}{1+4sR^{\alpha-1}}\le \frac{1}{(1+4sR^{\alpha-1})^{\alpha}}\le\left(\frac23\right)^{\frac16},$$which will yield a contradiction if $4s R^{\alpha-1}<\left(\frac32\right)^{\frac16}-1$ .

\end{proof}

By Lemma \ref{LocalFlow}, for each $i$, there exists a harmonic function $$h_i=h_{r_i}:A_{r_i-20s r_i^\alpha,r_i+20s r_i^\alpha}(p)\to\R$$ satisfying that for any $x\in A_{r_i-10s r_i^\alpha,r_i+10s r_i^\alpha}(p)$,
\begin{itemize}
	\item[(1)] $|h_i(x)-r_{p}(x)|\le s r_i^\alpha\Psi(s r_i^{\alpha-1},\delta|n)$, 
	
	\item[(2)] restricting on each $B_{10s r_i^\alpha}(x)$, $h_i|_{B_{10s r_i^\alpha}(x)}$ is a $\Psi(s r_i^{\alpha-1},\delta|n)$-splitting function,

\end{itemize}

Fix a smooth cut-off function $\phi:(-\infty,+\infty)\to[0,1]$ s.t. $\phi|_{[-3,3]}=1$, $\supp \phi\subset[-4,4]$, and $|\phi'|+|\phi''|\le100$. For $x\in A_{r_i-10s r_i^\alpha,r_i+10s r_i^\alpha}(p)$, define $\phi_i(x)=\phi(\frac{h_i(x)-r_i}{s r_i^\alpha})$. Specially, $\supp\phi_i\subset A_{r_i-5 s r_i^\alpha,r_i+5 s r_i^\alpha}(p)$ and $\phi_i|_{A_{r_i-2.5s r_i^\alpha,r_i+2.5s r_i^\alpha}(p)}\equiv1$. Hence $\phi_i$ can be extended on global $M$ canonically. Further, by Proposition \ref{GoodCoverProperties}, $\phi_i$ satisfies, for any $x\in M\backslash B_{r_0}(p)$, 
\begin{itemize}
	\item [(1)] there exists $i\ge0$, s.t. $\phi_i(x)=1$.
	
	\item [(2)] $D(x):=\sum_{i\ge0}\phi_i(x)\in[1,6]$.
\end{itemize}

For $x\in M\backslash B_{r_0}(p)$ and each $i\ge0$, put $\psi_i(x)=\frac{\phi_i(x)}{D(x)}$, and $h(x)=\sum_{i\ge0}\psi_i(x)h_i(x)$.

\textbf{A simple fact:} for some $x\in M\backslash B_{r_0}(p)$, $i\ge0$ and $c>0$, if $|r_p(x)-r_i|\le c s r_i^\alpha$, then $$|\frac{r_p(x)}{r_i}-1|= \frac{cs}{r_i^{1-\alpha}}\le cs R^{\alpha-1}\le \frac c{100},$$provided $sR^{\alpha-1}\le \frac1{100}$, which will be used below without specification sometimes.
\begin{proposition}\label{C^0-Closeness}
	For each $x\in M\backslash B_{r_0}(p)$, $|h(x)-r_p(x)|\le s r_p(x)^\alpha\Psi(s r_p(x)^{\alpha-1},\delta|n)$. Specially, $|\frac{h(x)}{r_p(x)}-1|\le s r_p(x)^{\alpha-1}\Psi(s r_p(x)^{\alpha-1},\delta|n)$.
	
\end{proposition}
\begin{proof}
For fixed $x\in M\backslash B_{r_0}(p)$, $$h(x)=\sum_{i\ge0}\psi_i(x)h_i(x)=\sum_{|r_p(x)-r_i|<10s r_i^\alpha}\psi_i(x)h_i(x).$$ So
	
	\begin{align*}
		&|h(x)-r_p(x)|\\\le& \sum_{|r_p(x)-r_i|<10s r_i^\alpha}\psi_i(x)|h_i(x)-r_p(x)|\\\le&\sum_{|r_p(x)-r_i|<10s r_i^\alpha}\psi_i(x)s r_i^\alpha\Psi(s r_i^{\alpha-1},\delta|n)\\\le&s r_p(x)^\alpha\Psi(s r_p(x)^{\alpha-1},\delta|n),
	\end{align*}
where the last inequality uses the simple fact mentioned above.

\end{proof}

\begin{proposition}\label{SupportFunctionReg}
	$$s r_i^{\alpha}\|\grad\psi_i\|+s^2r_i^{2\alpha}|\lap\psi_i|\le C(n).$$
\end{proposition}
\begin{proof}
	It's a direct calculation.
\end{proof}

The following proposition finishes the proof of Theorem \ref{MorseFunctionExistence}.
\begin{proposition}
	For any $x\in M\backslash B_{r_0}(p)$,
	$$\|\grad h(x)\|\le C(n),$$
	
	$$\dashint_{B_{s r_p(x)^\alpha}(x)}|\|\grad h\|^2-1|\le \Psi(s r_p(x)^{\alpha-1},\delta|n),$$
	$$s r_p(x)^\alpha|\lap h(x)|\le\Psi(s r_p(x)^{\alpha-1},\delta|n).$$
\end{proposition}
\begin{proof}
	For any fixed $x\in M\backslash B_{r_0}(p)$, by Proposition \ref{GoodCoverProperties} (1), there exists a $j\ge0$ s.t. \begin{equation}\label{r_p(x)-r_j}
		|r_p(x)-r_j|\le2.5s r_j^\alpha,
	\end{equation}
	For any $y\in B_{2s r_j^\alpha}(x)$ and any $i$ s.t. $|r_p(y)-r_i|\le 10s r_i^\alpha$, we have
	$$|r_p(y)-r_j|\le|r_p(y)-r_p(x)|+|r_p(x)-r_j|\le4.5 sr_j^\alpha.$$So this allows us to use the closeness properties of $h_i(y)$, $h_j(y)$ and $r_p(y)$,
	\begin{equation*}
			|h_i(y)-h_j(y)|\le|h_i(y)-r_p(y)|+|r_p(y)-h_j(y)|\le s r_p(y)^\alpha\Psi(s r_p(y)^{\alpha-1},\delta|n).
	\end{equation*}
	Again using the simple fact mentioned above, we see $|\frac{r_p(x)}{r_p(y)}-1|\le \frac12$. So 
	\begin{equation}\label{C^0Close}
		|h_i(y)-h_j(y)|\le s r_j^\alpha\Psi(s r_p(x)^{\alpha-1},\delta|n).
	\end{equation}
	
	Since $h_i,h_j$ are harmonic on $B_{2s r_j^\alpha}(x)$, by Cheng-Yau's gradient estimate, 
	\begin{equation}\label{C^1Close}
		\sup_{B_{s r_j^\alpha}(x)}\|\grad h_i-\grad h_j\|\le  \Psi(s r_p(x)^{\alpha-1},\delta|n).
	\end{equation}
	
	Now for any $y\in B_{s r_j^\alpha}(x)$, combining (\ref{C^0Close}), (\ref{C^1Close}) and Proposition \ref{SupportFunctionReg},
	\begin{align*}
		&\|\grad h(y)-\grad h_j(y)\|\\\le&\sum_{|r_p(y)-r_i|\le 10s r_i^\alpha}\|\grad\psi_i(y)\||h_i(y)-h_j(y)|+\psi_i(y)\|\grad h_i(y)-\grad h_j(y)\|\\\le&\sum_{|r_p(y)-r_i|\le 10s r_i^\alpha}C(n)(s r_i^\alpha)^{-1}s r_j^\alpha\Psi(s r_p(x)^{\alpha-1},\delta|n)+\psi_i(y) \Psi(s r_p(x)^{\alpha-1},\delta|n)\\\le&\Psi(s r_p(x)^{\alpha-1},\delta|n).
	\end{align*}
	Combining the above estimate and the fact that $h_j|_{B_{s r_j^\alpha}(x)}$ is a $\Psi(s r_j^{-1},\delta|n)$-splitting function, yields, 
	\begin{equation*}
		\dashint_{B_{s r_j^\alpha}(x)}|\|\grad h\|^2-1|\le \Psi(s r_p(x)^{\alpha-1},\delta|n).
	\end{equation*}
	Thus, combing (\ref{r_p(x)-r_j}) and relative volume comparison, $$\dashint_{B_{s r_p(x)^\alpha}(x)}|\|\grad h\|^2-1|\le\frac{\vol(B_{s r_j^\alpha}(x))}{\vol(B_{s r_p(x)^\alpha}(x))}   \dashint_{B_{s r_j^\alpha}(x)}|\|\grad h\|^2-1|\le \Psi(s r_p(x)^{\alpha-1},\delta|n).$$
	Finally, for any $y\in B_{s r_j^\alpha}(x)$,
	\begin{align*}
		&|\lap h(y)|=|\lap (h(y)-h_j(y))|\\\le&\sum_{|r_i-r_p(y)|<10s r_i^\alpha}|h_i(y)-h_j(y)||\lap\psi_i(y)|+2\|\grad\psi_i(y)\|\|\grad h_i(y)-\grad h_j(y)\|\\\le&C(n)\sum_{|r_i-r_p(y)|<10s r_i^\alpha}s r_p(y)^\alpha\Psi(s r_p(x)^{\alpha-1},\delta|n) (s r_p(y)^\alpha)^{-2} +(s r_p(y)^\alpha)^{-1}\Psi(s r_p(x)^{\alpha-1},\delta|n)  \\\le&(s r_p(x)^\alpha)^{-1} \Psi(s r_p(x)^{\alpha-1},\delta|n).
	\end{align*}

\end{proof}

\section{Finite Topological Type}

In this section, we show the function $h$ constructed in Section 3 is a Morse function. Precisely, we need the non-degeneracy Lemma \ref{NonDegenrateLemma} for $h$, which is a simple corollary of transformation theorems for almost splitting maps. Firstly, we point out that the following transformation theorem holds, which generalizes \cite[Theorem 1.7]{HH22} slightly. Recall that $B_r(x)$ is called $(\delta,K)$-Euclidean if there exists a metric space $(\R^K\times X,(0^K,x^*))$ s.t. $d_{GH}(B_r(x),B_r(0^K,x^*))\le r\delta$.

\begin{theorem}[Transformation Theorem]\label{TransThm}
	There exists $\delta_{0}=\delta_0(n, \epsilon, \eta,L)$ s.t, for any $\delta\in(0,\delta_{0})$, the following holds.
	Suppose $(M,p)$ is an $n$-manifold with $\Ric\ge-(n-1)\delta$ and $\overline{B_8(p)}$ is compact, and there exists $s\in(0,1)$ such that, for any $r\in[s,8]$, $B_{r}(p)$ is $(\delta, K)$-Euclidean but not $(\eta, K+1)$-Euclidean.
	Let $u : B_{8}(p)\rightarrow\mathbb{R}^{k}$ be a $C^2$-map, $1\le k\le K\le n$, satisfying for any $\alpha,\beta=1,..,k$,
	\begin{enumerate}
		\item [(1)] $\dashint_{B_8(p)}|\spa{\grad u^\alpha,\grad u^\beta}-\delta^{\alpha\beta}|\le\delta$,
		
		\item [(2)] $|\lap u^\alpha|_{C^0(B_8(p))}\le L$.
	\end{enumerate}
	Then for each $r\in[s,1]$, there exists a $k\times k$ lower triangle matrix $A_{r}=A_{r,p}$ with positive diagonal entries satisfying that,
	\begin{enumerate}
		
		\item[(3)] $\dashint_{B_r(p)}|\spa{\grad (A_ru)^\alpha,\grad (A_ru)^\beta}-\delta^{\alpha\beta}|\le\epsilon$,
		
		\item[(4)] $\dashint_{B_r(p)}\spa{\grad (A_ru)^\alpha,\grad (A_ru)^\beta}=\delta^{\alpha\beta}$,
		
		\item[(5)] $|A_{r}|+|A_r^{-1}|\le (1+\epsilon) r^{-\epsilon}$, here $|\cdot|$ means maximal norm of a matrix.
		
		\item[(6)] $|\lap (A_ru)^\alpha|_{C^0(B_r(p))}\le C(n,L)r^{-\epsilon}$, $\|\grad (A_ru)^\alpha\|_{C^0(B_r(p))}\le C(n,L)$.
	\end{enumerate}

\end{theorem}

Since \cite{CJN21}, there have been various generalized versions of transformation theorems for almost splitting functions, such as \cite{BNS,WZ21,HH22,HP22}. A slight modification of the proof of \cite[Theorem 1.7]{HH22} gives the proof of Theorem \ref{TransThm}. Note that the main difference between Theorem \ref{TransThm} and \cite[Proposition 7.7]{CJN21} is that, Theorem \ref{TransThm} does not assume the manifold is non-collapsed. In fact, the proof of \cite[Theorem 1.7]{HH22} is along the same lines as the one of \cite{CJN21}, except that \cite[Lemma 7.8]{CJN21} is replaced by \cite[Theorem 3.8]{HH22} (see Theorem \ref{thm-gap-har-RCD} below). For readers' convenience, we give the sketched proof in appendix.

As an immediate corollary of the above theorem, we have the following non-degeneracy lemma.
\begin{lemma}\label{NonDegenrateLemma}
	There exists $\delta=\delta(n,L)>0$, such that the following holds. Let $(M,p)$ be an $n$-manifold with $\overline{B_{1}(p)}$ compact and $u:B_1(p)\to\R$ satisfies properties 
	
	\begin{enumerate}
		\item [(1)] $\dashint_{B_1(p)}|\|\grad u\|^2-1|\le\delta$,
		
		\item [(2)] $|\lap u|_{C^0(B_1(p))}\le L$.
	\end{enumerate}
	
	If $M$ also satisfies 
	\begin{itemize}
		\item [(3)] $\Ric\ge-(n-1)\delta$, and $\vol(B_1(\tilde p))\ge(1-\delta)w_n$, where $\pi:(\widetilde{B_1(p)},\tilde p)\to (B_1(p),p)$ is the universal covering,

	\end{itemize}
	then $\myd u$ is non-degenerate at $p$.
\end{lemma}

\begin{proof}

	Note that $\tilde u:=u\circ\pi: B_{1}(\tilde p)\to\R$ satisfies properties 
	\begin{enumerate}
		\item [(1')] $\dashint_{B_1(\tilde p)}|\|\grad \tilde u\|^2-1|\le C(n)\delta$,
		
		\item [(2')] $\|\lap \tilde u\|_{C^0(B_1(\tilde p))}\le L$,
	\end{enumerate}
	where (1') is according to \cite[Lemma 1.6]{KW}. And the non-degeneracy of $\myd u(p)$ and $\myd\tilde u(\tilde p)$ coincide. Hence, w.l.g, we may assume $\vol(B_1(p))\ge(1-\delta)w_n$, and normalize $u(p)=0$. By Cheeger-Colding (\cite{CC96}), $d_{GH}(B_r(p),B_r(0^n))\le r\Psi(\delta|n)$ for any $r\in[0,1]$. Now by Theorem \ref{TransThm}, for any $r\in(0,\frac18)$, there exists a positive number, $A_r=A_{r,p}$, with $A_r\le r^{-\Psi(\delta|n)}$, s.t. $A_ru|_{B_{r}(p)}$ satisfies 
	\begin{equation*}\label{L^20517}
		\dashint_{B_r(p)}|\|\grad (A_ru)\|^2-1|\le \Psi(\delta|n,L).
	\end{equation*}
	
	Let $r_h\le 1$ be the $C^{1}$-harmonic radius at $p$, i.e., $r_h$ satisfies that, there exists a harmonic coordinate map, $x=(x^1,...,x^n):(B_{r_h}(p),p)\to(\R^n,0^n)$ satisfying that, for each $i$, $\lap x^i=0$ and the metric components $g_{ij}=g(\frac{\partial}{\partial x^i},\frac{\partial}{\partial x^j})$ satisfies that 
	\begin{equation}\label{C^1NormOfg}
		\|g_{ij}-\delta_{ij}\|_{C^0(B_{r_h}(p))}+r_h\|\frac{\partial g_{ij}}{\partial x^{k}}\|_{C^0(B_{r_h}(p))}\le 10^{-n}.
	\end{equation}
	Put $(\bar M,\bar p,\bar g,\bar x^i):=(M,p,r_h^{-2}g,r_h^{-1}x^i)$, and $\bar u:=\frac1{r_h}A_{r_h}u$. Then according to Theorem \ref{TransThm}, the restriction $\bar u:(B_{1}(\bar p),\bar p)\to(\R,0)$ satisfies,
	\begin{equation}\label{L^2gradient}
		\dashint_{B_1(\bar p)}|\|\bar \grad \bar u\|^2-1|\le\Psi(\delta|n,L).
	\end{equation}
	\begin{equation}\label{UseInElliptic}
	\|\bar\lap\bar u\|_{C^0(B_{1}(\bar p))}\le r_h^{1-\Psi(\delta|n)}L,\quad \|\bar\grad \bar u\|_{C^0(B_{1}(\bar p))}\le C(n,L).
	\end{equation}

Then by elliptic estimate, and (\ref{UseInElliptic}), we have,
	\begin{equation}\label{EllipticEstimate}
		\|\bar u \|_{C^{1,\frac12}(B_{\frac12}(\bar p))}\le C(n)(\|\bar\lap \bar u\|_{C^0(B_{1}(\bar p))}+\|\bar u\|_{C^\frac12(B_1(\bar p))})\le C(n,L).
	\end{equation}
	
	Put $A=\{\bar x\in B_{1}(\bar p)|\large|\|\bar \grad \bar u(\bar x)\|^2-1|>\sqrt{\Psi(\delta|n,L)}\}$. By (\ref{L^2gradient}), $\vol(A)\le \sqrt{\Psi(\delta|n,L)}\vol B_1(\bar p)$. We may assume $\bar p\in A$, otherwise the proof is finished. For any $B_{\eta}(\bar p)\subset A$, $\eta\le 1$, by volume comparison, we have $$\sqrt{\Psi(\delta|n,L)}\ge \frac{\vol B_{\eta}(\bar p)}{\vol B_1(\bar p)}\ge c(n)\eta^n ,$$which means $\eta\le c(n)^{-n}\Psi(\delta|n,L)^{\frac1{2n}}=:\Psi_1(\delta|n,L)$. So there exists $\bar q\in B_1(\bar p)\backslash A$ with $d_{\bar g}(\bar q,\bar p)\le \Psi_1(\delta|n)$. Then by estimate (\ref{EllipticEstimate}), $$\large|\frac{\partial \bar u}{\partial \bar x^i}(\bar p)-\frac{\partial \bar u}{\partial \bar x^i}(\bar q)\large|\le C(n,L)d_{\bar g}(\bar p,\bar q)^{\frac12}\le C(n,L)\sqrt{\Psi_1(\delta|n,L)}=:\Psi_2(\delta|n,L).$$Combining the above estimate and (\ref{C^1NormOfg}) yields,
	$$|\|\bar\grad \bar u(\bar p)\|^2-1|\le |\|\bar\grad \bar u(\bar q)\|^2-1|+C(n)\Psi_2(\delta|n,L)\le \sqrt{\Psi(\delta|n,L)}+C(n)\Psi_2(\delta|n,L),$$which implies $\myd u(p)$ is not degenerate provided $\delta$ is small depending on $n,L$. 
\end{proof}
\begin{remark}\label{NondegForOtherCase}
	If condition (3) in Lemma \ref{NonDegenrateLemma}, is replaced by one of the following conditions,
	\begin{enumerate}
	\item [(4)] $\Ric\ge-(n-1)\delta$, and conjugate radius $r_{c}\ge 1$,
	\item [(5)] Sectional curvature $K\ge-\delta$,
	
	\end{enumerate}
	
	then the same conclusion holds.
	
	In fact, for condition (4): According to \cite{AC92}, the $C^{0,\alpha}$-norm of pullback metric on tangent space $(T_pM,0_p)$ is uniformly bounded from above by $C(n)>0$. Specially, there exists $r(n)\in(0,1)$ s.t. $d_{GH}(B_{r(n)}(0_p),B_{r(n)}(0^n))\le\Psi(\delta|n)r(n)$. Then the remaining proof is the same as the case (3). For condition (5), \cite[Theorem 1.1]{HH22} asserts that for manifolds with $K\ge-\delta$, any harmonic $\delta$-splitting function $u:B_1(p)\to\R$ is non-degenerate at $p$. Indeed the condition that $u$ is harmonic is not essential by Theorem \ref{TransThm} above. That is, if $K\ge-\delta$, and $u:B_1(p)\to\R$ satisfies (1),(2) in Lemma \ref{NonDegenrateLemma}, for $\delta\le \delta(n,L)$, then $\myd u$ is non-degenerate at $p$.
	
\end{remark}

Now we are ready to prove Theorem \ref{FiniteToplgType-Main} and Corollary \ref{Corrollaries-FinitTopType}.

\subsection{Proof of Theorem \ref{FiniteToplgType-Main}}\quad

	By Theorem \ref{MorseFunctionExistence}, there exist $R>0$, and a smooth function $h:M\backslash B_R(p)\to\R$ satisfying that for any $x\in M\backslash B_R(p)$,
	
	\begin{itemize}
		\item [(1)] $\frac{|h(x)-r_p(x)|}{s r_p(x)^{\alpha}}\le\Psi(s R^{\alpha-1},\delta|n)$,
		
		\item [(2)] Restriction $h:B_{s r_p(x)^\alpha}(x)\to\R$ is a $\Psi(s R^{\alpha-1},\delta|n)$-splitting function.
		
		\item [(3)] $s r_p(x)^\alpha|\lap h(x)|\le \Psi(s R^{\alpha-1},\delta|n)$. 
		
	\end{itemize}

	For each $x\in M\backslash B_R(p)$, by properties (2), (3) of $h$ and (\ref{AlmostMaximalRewindingVolume}), applying Lemma \ref{NonDegenrateLemma} to $$u:= \frac{1}{sr_p(x)^{\alpha}}h|_{B_1(x,(sr_p(x)^\alpha)^{-1}d)}:B_1(x,(sr_p(x)^\alpha)^{-1}d)\to\R,$$ shows $h$ is non-degenerate at $x$, provided $sR^{\alpha-1},\delta$ sufficiently small depending only on $n$. Combining property (1) of $h$, there exists a large $C>0$, s.t., $h|_{h^{-1}(C,\infty)}:h^{-1}(C,\infty)\to(C,\infty)$ is non-degenerate everywhere and proper. This implies $h|_{h^{-1}(C,\infty)}:h^{-1}(C,\infty)\to(C,\infty)$ is a trivial fiber bundle with compact fiber. Hence $M$ is diffeomorphic to the bounded domain $M\backslash h^{-1}[2C,\infty)$.
\begin{remark}\label{RemarkOnNonDeg}
	Note that the construction of $h$ in Theorem \ref{MorseFunctionExistence} only involves the growth conditions of excess function. Hence due to Lemma \ref{NonDegenrateLemma} and Remark \ref{NondegForOtherCase}, our approach to Theorem \ref{FiniteToplgType-Main} also applies in the case of replacing local almost maximal rewinding volume by correspondent assumptions on sectional curvature or conjugate radius bound.
\end{remark}

\subsection{Proof of Corollary \ref{Corrollaries-FinitTopType}}\quad

(1) By definition of Euclidean volume growth, there exists $\nu>0$ s.t. $\lim_{r\to\infty}\frac{\vol(B_r(p))}{r^n}= \nu$. Hence by volume comparison, $\frac{\vol(B_r(p))}{r^n}-\nu\to 0+$ as $r\to\infty$. Now applying Proposition \ref{FiniteToplgType-VariousGrowthCondition}(3) for $\alpha=1$, we see that, for any $\delta>0$, for $x\in M$ outside a compact set, $\frac{e_p(x)}{r_p(x)}\le \delta$. We claim that for any $\epsilon>0$, there exists an $s=s(\epsilon,M)\in(0,1)$ s.t. for any $x\in M$ outside a compact subset, $$\vol(B_{sr_p(x)}(x))\ge (1-\epsilon) w_n\left(sr_p(x)\right)^n.$$ Then required conclusion follows by Theorem \ref{FiniteToplgType-Main} for $\alpha=1$.

Argue by contradiction; suppose there exist $\epsilon_0>0$, $s_i\to0$, $x_i\in M$ with $r_p(x_i)\to\infty$, s.t. $\vol(B_{s_ir_p(x_i)}(x_i))< (1-\epsilon_0) w_n\left(s_ir_p(x_i)\right)^n$. Put $r_i=r_p(x_i)$ and $(M_i,d_i)=(M,r_i^{-1}d)$. Passing to a subsequence, we may assume $(M_i,p,x_i)\GH(X,p_\infty,x_\infty)$. Since $M$ has Euclidean volume growth, by Theorem \ref{Thm-VolumeCone}, $X$ is a metric cone $C(Y)$ with cross section $Y$ of $\diam(Y)\le\pi$. By conditions, $Y$ is a smooth manifold. Specially, there exists $s_0>0$, s.t., $\vol(B_{s_0}(x_\infty))\ge(1-\frac12\epsilon_0)w_ns_0^n$. Using Colding's volume convergence theorem (\cite{Co97}), for $i$ large, $\vol(B_{s_0}(x_i,M_i))\ge (1-\frac23\epsilon_0)w_ns_0^n$. For $i$ large, s.t. $s_i<s_0$, by volume comparison, $\vol(B_{s_i}(x_i,M_i))\ge (1-\frac23\epsilon_0)w_ns_i^n$. That is, $\vol(B_{s_ir_i}(x_i))\ge (1-\frac23\epsilon_0)w_n(s_ir_i)^n$ which contradicts to the contradicting assumption.

(2) Applying Theorem \ref{FiniteToplgType-Main} for $\alpha=0$, and $s=\rho$ and combining Proposition \ref{FiniteToplgType-VariousGrowthCondition}(1) or (2) give the proof.

(3)  We claim that for any $\epsilon>0$, there exists an $s=s(\epsilon,M)\in(0,1)$ s.t. for any $x\in M$ outside a compact subset, $$\vol(B_{s}(x))\ge (1-\epsilon) w_ns^n.$$ Then Proposition \ref{FiniteToplgType-VariousGrowthCondition} (1) or (3) gives the excess growth condition (\ref{ExcessGrowthCondition}) for $\alpha=0$ and $s$ which gives the proof by applying Theorem \ref{FiniteToplgType-Main}. 

In below we prove the claim.

\textbf{The volume case:} $\frac{\vol( B_r(p))}{r^4}=\nu+o(r^{-\frac94})$, for some $\nu>0$.

By the volume growth conditions and relative volume comparison, obvious for any $x\in M$, $\vol(B_1(x))\ge \nu>0$. We argue by contradiction; if there exist $\epsilon_0\in(0,1)$, $s_i\to0$, $x_i\in M$ with $r_p(x_i)\to\infty$ s.t. $\vol (B_{s_i}(x_i))<(1-\epsilon_0)w_4s_i^n$. Passing to a subsequence, we may assume $(M,x_i)\GH(X,x_\infty)$, where the convergence is non-collapsed. Note that by Proposition \ref{FiniteToplgType-VariousGrowthCondition} (3), $e_p(x_i)\le \delta_i\to0$, hence by quantitative splitting theorem \ref{Quantitativesplitting}, $X$ splits a line, i.e., $X=\R\times Y$, where $Y$ is a length space of Hausdorff dimension $3$. According to Cheeger-Naber's codimension $4$ theorem (\cite{CN}), the subset of singular points of $X$ has at most dimension $0$. This shows $Y$ contains no singular points, so does $X$. Specially, there exists $s_0>0$, s.t. $\vol(B_{s_0}(x_\infty))\ge (1-\frac12\epsilon_0)w_4s_0^4$. Using Colding's volume convergence theorem (\cite{Co97}), for $i$ large, $\vol(B_{s_0}(x_i))\ge (1-\frac23\epsilon_0)w_4s_0^4$. For $i$ large, s.t. $s_i<s_0$, by volume comparison, $\vol(B_{s_i}(x_i))\ge (1-\frac23\epsilon_0)w_4s_i^4$ which yields a contradiction.

\textbf{The diameter case:} $\sup_{r>0}\D_p(r)<\infty$.

By Lemma \ref{SmallEsentialDiamterImpliesSmallExcess}, we have $M$ is isometric to $\R\times N$ with $N$ compact, or $\sup\diam (\partial B_r(p))<\infty$. It suffices to prove in the latter case.

Up to a scaling, we may assume $\diam (\partial B_r(p))\le \frac18$ and $0\le \Ric_M\le C<\infty$.

Claim (A): there exists $v>0$, s.t. for any $x\in M$, $\vol(B_1(x))\ge v>0$. Let $\gamma$ be a ray with unit speed from $p$ and $b_\gamma(x)=\lim_{t\to\infty}\{t-d(x,\gamma(t))\}$ be the Busemann function associated to $\gamma$. By excess estimate, $0\le r_p(x)-b_\gamma(x)\le \frac{1}{r_p(x)^{\frac{1}{3}}}$ for all $x$ outside a compact subset. Combining the fact $\diam (\partial B_r(p))\le \frac18$, for some large $t_0>0$, we have $b_{\gamma}^{-1}(t-0.5,t+0.5)\subset B_{1}(\gamma(t))$ for all $t\ge t_0$. Applying \cite[Lemma 20]{Sor98}, one concludes $\vol(B_1(\gamma(t)))\ge \vol(b_{\gamma}^{-1}(t-0.5,t+0.5))\ge \vol(b_{\gamma}^{-1}(t_0-0.5,t_0+0.5))$. Put $v:=\min\{\vol(b_{\gamma}^{-1}(t_0-0.5,t_0+0.5)),\inf_{b_\gamma(x)\le t_0}\vol(B_1(x))\}$. For any $x\in M$, $B_{2}(x)\supset B_1(\gamma(r_p(x)))$. Now by volume comparison, $$\vol(B_1(x))\ge\frac1{16}\vol(B_2(x))\ge  \frac1{16}\vol(B_1(\gamma(r_p(x))))\ge \frac{v}{16}.$$ Hence the Claim (A) follows. The remain is similar as the volume case.

Now the proof of Corollary \ref{Corrollaries-FinitTopType} is complete.

\begin{remark}
	There is another proof of Corollary \ref{Corrollaries-FinitTopType} (1) which verifies the non-degeneracy of function $$b(\cdot):=\left(\frac{\nu}{\omega_n}G(p,\cdot)\right)^{\frac{1}{2-n}},$$ introduced by Colding-Minicozzi in \cite{CoM97}, instead of $h$, where $G(x,y)$ is the minimal positive Green's function on $M$. In fact, by \cite{CoM97}, under the Euclidean volume growth condition, we have
	\begin{equation*}
		|\grad b|\le 1,
	\end{equation*}
	and for each $\delta>0$, there exists $R_0>0$ s.t., for all $R\ge R_0$,
	\begin{equation*}
		\sup_{x\in B_R(p)} R^{-1}|b(x)-r_p(x)|+\dashint_{b\le R}|\|\grad b\|^2-1|^2+\dashint_{b\le R}|\hess(b^2)-2g|^2<\delta.
	\end{equation*}
Then applying Lemma \ref{NonDegenrateLemma} to $b$ finishes the proof.
\end{remark}

\section{appendix}

\subsection{Proof of Lemma \ref{SmallVolumeGrowth}}

Let $\{B_1(y_1),...B_1(y_N)\}$ be a maximal set of disjoint balls with $y_1,...,y_N\in\Sigma_r$. Then $\Sigma_r\subset\cup _{i=1}^NB_2(y_i)\subset A_{r-2,r+2}(p)$. Hence, $$N\le\frac{\vol(A_{r-2,r+2}(p))}{v_p(r)}.$$

Since $\Sigma_r$ is connected, for any $x_1,x_2\in\Sigma_r$, let $\gamma:[0,1]\to\Sigma_r$ be a curve connecting $x_1,x_2$. Now we find a path from $x_1$ to $x_2$ as follows. Putting $t_0=0$, choose an arbitrary $B_2(y_{i_0})$ s.t. $x_1=\gamma(t_0)\in B_2(y_{i_0})$, and put $t_1=\sup\{t\in[0,1]|\gamma(t)\in B_2(y_{i_0})\}$. Assuming $t_k\in[0,1)$ is defined, choose an arbitrary $B_2(y_{i_k})$ s.t. $\gamma(t_k)\in B_2(y_{i_k})$, and put $t_{k+1}=\sup\{t\in[0,1]|\gamma(t)\in B_2(y_{i_k})\}$. Note that $t_k<t_{k+1}$. Repeat this procedure provided $t_k<1$. Hence we obtain $0=t_0<t_1<...<t_K=1$ and distinct $B_2(y_{i_0}),...,B_2(y_{i_{K-1}})$ s.t. $\gamma(t_k),\gamma(t_{k+1})\in \overline{B_{2}(y_{i_k})}$, $k=0,...,K-1$. Now we have $$d(x_1,x_2)\le\sum_{k=0}^{K-1} d(\gamma(t_k),\gamma(t_{k+1}))\le 4K\le 4N\le 4\frac{\vol(A_{r-2,r+2}(p))}{v_p(r)}.$$ 

\subsection{Proof of Lemma \ref{LargeVolumeGrowth}}

Let $x\in\partial B_r(p)\backslash R_p$ and $h>0$ s.t. $B_h(x)\cap R_p=\emptyset$. Note that $h\le r$ and $B_h(x)\cup R_p(r-h,r+h)\subset A_{r-h,r+h}(p)$, where $R_p(r-h,r+h)=\{y\in R_p|r-h<r_p(y)<r+h\}$. Then $$\vol(B_h(x))+\vol(R_p(r-h,r+h))\le \vol(A_{r-h,r+h}(p)).$$
By relative volume comparison, $$\vol(B_h(x))\ge \nu h^n,$$
$$\frac{\vol(R_p(r-h,r+h))}{(r+h)^n-(r-h)^n}\ge \frac{\vol(R_p(r,R))}{R^n-r^n}\to\nu,\, \text{as }R\to\infty,$$
$$\frac{\vol(A_{r-h,r+h}(p))}{(r+h)^n-(r-h)^n}\le \frac{\vol(B_r(p))}{r^n},$$ where the convergence of the second line is by the fact $\frac{\vol(B_R(p)\cap R_p)}{R^n}\to\nu$ (see \cite[Lemma 4 ]{OSY}). Hence, $$ h^n\le\nu^{-1} \left(\frac{\vol(B_r(p))}{r^n}-\nu\right)\left((r+h)^n-(r-h)^n\right).$$Using the fact that $(r+h)^n-(r-h)^n\le 2n h(r+h)^{n-1}$ gives the desired inequality.

\subsection{Sketched Proof of Theorem \ref{TransThm}}

It just need to prove for $\epsilon\le\epsilon(n,\eta,L)$, which will be specified later.

Firstly, according to \cite[Corollary 3.3]{HP22}, $\|\grad u^\alpha\|_{C^0(B_2(p))}\le C(n,L)$.

Let $S_\delta(\epsilon)$ be the set of those $r\in[s,1]$ s.t. there exits a $k\times k$ lower triangle matrix $A_r$ satisfying (3) and (4). 

\begin{sublemma}\label{SimplieFact}
	For any $s'\in(0,1)$, if $\delta\le C(n,s',L)^{-1}\epsilon$ for some $C(n,s',L)\ge1 $, then $[s',1]\subset S_\delta(\epsilon)$. 
\end{sublemma}
\begin{proof}
	For $r\in[s',1]$, putting $B_r^{\alpha\beta}:=\dashint_{B_r(p)}\spa{\grad u^\alpha,\grad u^\beta}$, by volume comparison,
	\begin{align*}
		|B_r^{\alpha\beta}-\delta^{\alpha\beta}|\le\dashint_{B_r(p)}|\spa{\grad u^\alpha,\grad u^\beta}-\delta^{\alpha\beta}|\le\frac{\vol(B_4(p))}{\vol (B_r(p))}\dashint_{B_4(p)}|\spa{\grad u^\alpha,\grad u^\beta}-\delta^{\alpha\beta}|\le C(n)s'^{-n}\delta
	\end{align*}
	which implies $B_r^{\alpha\beta}$ is positive defined provided $\delta$ is small depending on $n,s'$. By Cholesky decomposition, there exists a unique $k\times k$ lower triangle matrix $A_r$ with positive diagonal entries s.t. $A_rB_rA_r^T=I_k$. It's not hard to see $|A_r-I_k|\le C(n,s')\delta$. Hence $A_r$ satisfies Theorem \ref{TransThm} (3) and (4).
\end{proof}

According to the above sublemma, for $\delta\le C(n,L,0.1)^{-1}\epsilon$, $[\frac1{10},1]\subset S_\delta(\epsilon)$. Specially, $S_\delta(\epsilon)\neq\emptyset$. Let $\bar s=\inf S_\delta(\epsilon)$.

\begin{sublemma}\label{GrowthOfMatrixNorm220515}
	
	For $\bar s\le t\le r\le 1$, $\alpha=1,...,k$, $$|A_tA_r^{-1}|+|A_rA_t^{-1}|\le(1+C(n)\epsilon)\left(\frac rt\right)^{C(n)\epsilon},$$$$|A_t|+|A_t^{-1}|\le (1+C(n)\epsilon)t^{-C(n)\epsilon},$$
	$$|\lap (A_ru)^\alpha|_{C^0(B_r(p))}\le C(n,L)r^{-C(n)\epsilon},\quad\|\grad (A_ru)^\alpha\|_{C^0(B_r(p))}\le C(n,L).$$
\end{sublemma}
\begin{proof}
	For $\bar s\le t\le r\le 1$, by volume comparison,
	\begin{align*}
		\dashint_{B_t(p)}|\spa{\grad (A_ru)^\alpha,\grad (A_ru)^\beta}-\delta^{\alpha\beta}|\le\frac{\vol(B_r(p))}{\vol(B_t(p))}\dashint_{B_r(p)}|\spa{\grad (A_ru)^\alpha,\grad (A_ru)^\beta}-\delta^{\alpha\beta}|\le C(n)\epsilon\left(\frac r t\right)^n.
	\end{align*}
	
	Now assuming $t\le r\le 2t$, for $\epsilon$ is small depending on $n$, by Cholesky decomposition, there exists a unique lower diagonal matrix with positive diagonal entries $\bar A$ with $|\bar A-I_k|\le C(n)\epsilon$, s.t. $$\bar A^{\alpha\gamma_1}\dashint_{B_t(p)}\spa{\grad (A_ru)^{\gamma_1},\grad (A_ru)^{\gamma_2}}(\bar A^T)^{\gamma_2\beta}=\delta^{\alpha\beta}.$$ 
	By definition, we also have, $$A_t^{\alpha\gamma_1}\dashint_{B_t(p)}\spa{\grad u^{\gamma_1},\grad u^{\gamma_2}}A_t^{\gamma_2\beta}=\delta^{\alpha\beta}.$$ So by uniqueness of Cholesky decomposition, $A_t=\bar A A_{r}$. Thus, 
	\begin{equation}\label{Key220515}
		|A_tA_{r}^{-1}-I_k|\le  C(n)\epsilon
	\end{equation}

	Now for any $\bar s\le t\le r\le 1$, $m=0,1,2,...,\bar m$, s.t. $2^{\bar m}t\le r<2^{\bar m+1}t$, combining (\ref{Key220515}) and the elemental fact that for $k\times k$ matrices, $|AB-I_k|\le |A-I_k|+|B-I_k|+k|A-I_k||B-I_k|$, yields $$|A_tA_{2^mt}^{-1}-I_k|\le (1+kC(n)\epsilon)^m-1.$$Thus, $$|A_tA_r^{-1}|=|A_tA_{2^{\bar m}t}^{-1}(A_{2^{\bar m}t}A_r^{-1}-I_k)+A_tA_{2^{\bar m}t}^{-1}|\le (1+C(n)\epsilon)\left(\frac rt\right)^{C(n)\epsilon}.$$The proofs of $|A_rA_t^{-1}|$ and $|A_t^{-1}|$ are similar. 
	
	Now using the estimate of $A_r$,
	\begin{equation}\label{ineq-LapEs}
		|\lap (A_ru)^\alpha|_{C^0(B_r(p))}\le k|A_r|L\le (1+C(n)\epsilon)kLr^{-C(n)\epsilon},
	\end{equation}
	which is the third inequality.
	
	Put $(\bar M,\bar p,\bar g)=(M,p,\frac1{r^2}g)$. By (\ref{ineq-LapEs}), for $\bar s\le r\le\frac14$,$$|\bar\lap (\frac1rA_{4r}u)^\alpha|_{C^0(B_4(\bar p))}=r|\lap (A_{4r}u)^\alpha|_{C^0(B_{4r}(p))}\le C(n,L)r^{1-C(n)\epsilon}.$$
	Combining the above inequality and $\dashint_{B_4(\bar p)}|\|\bar\grad(\frac1rA_{4r}u)^\alpha\|^2-1|\le\epsilon$ (definition of $A_{4r}$), according to \cite[Corollary 3.3]{HP22} again, yields, $$\|\grad (A_{4r}u)^\alpha\|_{C^0(B_r( p))}=\|\bar\grad (\frac1rA_{4r}u)^\alpha\|_{C^0(B_1(\bar p))}\le C(n,L).$$ Hence 
	\begin{align*}
	\|\grad  (A_{r}u)^\alpha\|_{C^0(B_r(p))}=\|\grad  (A_{r}A_{4r}^{-1})^{\alpha\beta}(A_{4r}u)^\beta\|_{C^0(B_r(p))}&\le C(n)\sum_\beta\|\grad (A_{4r}u)^\beta\|_{C^0(B_r(p))}\\&\le C(n,L).
	\end{align*}

\end{proof}

The remain is to show $\bar s=s$.

Argue by contradiction; suppose there exist $\epsilon,\eta,L>0$, and a sequence of manifolds $(M_i,p_i)$, $u_i:(B_1(p_i),p_i)\to(\R^k,0^k)$, satisfying $\Ric_{M_i}\ge-(n-1)\delta_i\to0$, and, for each $i$, there is some $s_i\in(0,1)$ such that, for any $r\in[s_i,1]$, $B_{r}(p_i)$ is $(\delta_i, K)$-Euclidean but not $(\eta_0, K+1)$-Euclidean, and  
\begin{enumerate}
	\item [(1')] $\dashint_{B_8(p_i)}|\spa{\grad u_i^\alpha,\grad u_i^\beta}-\delta^{\alpha\beta}|\le\delta_i$,
	
	\item [(2')] $|\lap u_i^\alpha|_{L^\infty(B_8(p_i))}\le L$,
\end{enumerate}
and $\bar s_i:=\inf S_{\delta_i}(\epsilon)>s_i$. By Sublemma \ref{SimplieFact}, $\bar s_i\to0$. Set $(\bar M_i,\bar p_i)=(\bar s_i^{-1}M_i,p_i)$ and $\bar u_i=s_i^{-1}A_{\bar s_i}u_i:(B_{\bar s_i^{-1}}(\bar p_i),\bar p_i)\to(\R^k,0^k)$. By definition of $A_{\bar s_i}$, $\bar u_i$ satisfies,
\begin{enumerate}
	\item[(3')] $\dashint_{B_1(\bar p_i)}|\spa{\bar \grad \bar u_i^\alpha,\bar \grad \bar u_i^\beta}-\delta^{\alpha\beta}|\le\epsilon$,
	
	\item[(4')] $\dashint_{B_{1}(\bar p_i)}\spa{\bar \grad \bar u_i^\alpha,\bar \grad \bar u_i^\beta}=\delta^{\alpha\beta}$,
	
\end{enumerate}

and by Sublemma \ref{GrowthOfMatrixNorm220515}, 
\begin{enumerate}
	\item[(5')] $|A_{\bar s_i}|\le \bar s_i^{-C(n)\epsilon}$.
	
	\item[(6')] $|\bar \lap \bar u_i^\alpha(\bar x)|_{C^0(B_1(\bar p_i)))}\le C(n,L)\bar s_i^{1-C(n)\epsilon}$.
	
	\item[(7')] $\|\bar \grad \bar u_i^\alpha\|_{C^0(B_{1}(\bar p_i))}\le C(n,L)$.
\end{enumerate}

For $x\in B_{1}(p_i)\backslash B_{\bar s_i}( p_i)$, putting $r=d_{M_i}(x, p_i)$, and $\bar x\in\bar M_i$ be the identity image of $x$, we have
\begin{align*}
	&\|\bar \grad \bar u_i^{\alpha}(\bar x)\|_{\bar g_i}=\|\grad (A_{\bar s_i}u_i)^\alpha(x)\|_{g_i}= \|\grad (A_{\bar s_i}A^{-1}_{r})^{\alpha\beta}(A_{r}u_i)^\beta(x)\|_{g_i}\\\le& \left(\frac{r}{\bar s_i}\right)^{C(n)\epsilon}\sum_{\beta}\|\grad (A_ru_i)^{\beta}(x) \|_{g_i}\le C(n,L)\left(\frac{r}{\bar s_i}\right)^{C(n)\epsilon}=C(n,L)d_{\bar M_i}(\bar x,\bar p_i)^{C(n)\epsilon}.
\end{align*}

Combining (7') and the above estimate, for any $\bar x\in B_{\bar s_i^{-1}}(\bar p_i)$, $$\|\bar \grad \bar u_i(\bar x)\|\le C(n,L)d_{\bar M_i}(\bar x,\bar p_i)^{C(n)\epsilon}+C(n,L),$$which implies $$| \bar u_i(\bar x)|\le C(n,L)d_{\bar M_i}(\bar x,\bar p_i)^{1+C(n)\epsilon}+C(n,L).$$

Passing to a subsequence, we may assume $(\bar M_i,\bar p_i)\GH(X,p_\infty)$ and by (6') and (7'), $\bar u_i\myarrow{W^{1,2}}\bar u:X\to\R^k$. By assumption of $M_i$, $X$ is isometric to $\R^K\times Y$, and for any $R\ge1$, $X$ is not $(\frac12\eta,K+1)$-Euclidean. Further, $\bar u$ is a harmonic function, satisfying for any $\bar x\in X$,$$| \bar u(\bar x)|\le C(n,L)d_{X}(\bar x,\bar p)^{1+C(n)\epsilon}+C(n,L).$$ According the following gap theorem, 

\begin{theorem}[\cite{HH22}]\label{thm-gap-har-RCD}
	Given $\eta> 0$ and $N\geq k\geq1$ with $k\in \mathbb{Z}^{+}$, there exists $\epsilon=\epsilon(N,\eta)>0$ such that the following holds.
	Suppose $(X,d,m)$ is an $\mathrm{RCD}(0,N)$ space which is $k$-splitting, and $(B_{r}(p),d)$ is not $(\eta,k+1)$-Euclidean for any $r\geq 1$.
	If $u : X\to \mathbb{R}$ is a harmonic function such that
	\begin{align*}
		|u(x)|\leq Cd(x,p)^{1+\epsilon}+ C,
	\end{align*}
	for some $C> 0$,
	then $u$ is a linear combination of the $\mathbb{R}^{k}$-coordinates in $X$.
\end{theorem}

If one assumes $\epsilon\le \epsilon(n,\eta)$ at beginning, $\bar u$ is a linear function. And by $W^{1,2}$-convergence, $$\dashint_{B_{1}(p_\infty)}\spa{\grad \bar u^\alpha, \grad \bar u^\beta}=\delta^{\alpha\beta},$$which implies $\bar u^1,...,\bar u^k$ is an orthonormal basis of some $\R^k$-factor. Hence 
\begin{align*}
	4\lim_{i\to\infty}\dashint_{B_1(\bar p_i)}|\spa{\bar \grad \bar u_i^\alpha,\bar \grad \bar u_i^\beta}-\delta^{\alpha\beta}|=\dashint_{B_1( p_\infty)}|\|\grad \bar u^\alpha+\grad \bar u^\beta\|^2-\|\grad \bar u^\alpha-\grad \bar u^\beta\|^2-4\delta^{\alpha\beta}|=0.
\end{align*} 

So for large $i$, $$\dashint_{B_1(\bar p_i)}|\spa{\bar \grad \bar u_i^\alpha,\bar \grad \bar u_i^\beta}-\delta^{\alpha\beta}|\le\epsilon_i\to0.$$Combining the above inequality and inequality (6'), according to Sublemma \ref{SimplieFact}, we have for any $r\in[\frac1{10},1]$, there exists a $k\times k$ lower triangle matrix $\bar A_{i,r}$ satisfying properties (3) and (4) in Theorem \ref{TransThm} respect to $\bar u_i$. So putting $A_{r\bar s_i}:=A_{i,r}A_{\bar s_i}$, for each $r\in[\frac1{10},1]$, $$\dashint_{B_{r\bar s_i}(p_i)}|\spa{ \grad (A_{r\bar s_i} u_i)^\alpha, \grad  (A_{r\bar s_i} u_i)^\beta}-\delta^{\alpha\beta}|\le\epsilon,$$ and $$\dashint_{B_{r\bar s_i}(p_i)}\spa{\grad (A_{r\bar s_i}u_i)^\alpha,\grad (A_{r\bar s_i}u_i)^\beta}=\delta^{\alpha\beta},$$which means $[\frac{1}{10}\bar s_i,1]\subset S_{\delta_i}(\epsilon)$. This yields a contradiction to definition of $\bar s_i$.

\quad

\noindent\textbf{Acknowledgments.}
The author would like to thank Prof. Xian-Tao Huang for many helpful discussions and suggestions. The author is also grateful to Prof. Xiaochun Rong and Prof. Xi-Ping Zhu for their constant encouragement. The author is partially supported by China Postdoctoral Science Foundation 2021M693675 and Guang-dong Natural Science Foundation 2022A1515011072.

\bibliographystyle{alpha}
\bibliography{ref}

\end{document}